\documentclass[11pt,a4paper,leqno]{amsart}

\usepackage[english]{babel}
\usepackage[latin1]{inputenc}
\usepackage[hidelinks]{hyperref}

\usepackage{
amsmath,
latexsym,
amssymb,
math rsfs,
amsthm,
graphicx,
xcolor,
dsfont,
enumerate}

\usepackage{geometry}
\newtheorem{teo}{Theorem}[section]
\newtheorem{lem}{Lemma}[section]
\newtheorem{pro}{Proposition}[section]
\newtheorem{cor}{Corollary}[section]
\newtheorem{example}{Example}[section]

\theoremstyle{remark}
\newtheorem{rem}{Remark}[section]

\newcommand{\xv}{\mathbf{x}}
\newcommand{\yv}{\mathbf{y}}
\newcommand{\R}{\mathbb{R}}
\newcommand{\Z}{\mathbb{Z}}
\newcommand{\I}{\mathcal{I}}

\newcommand{\NN}{\mathcal{N}}
\newcommand{\ham}{H_{\alpha}}
\newcommand{\one}{\mathds{1}}
\newcommand{\dom}{\mathscr{D}}

\begin{document}
\title[ The action of Volterra integral operators with highly singular kernels]{ The action of Volterra integral operators with highly singular kernels on H\"older continuous, Lebesgue and Sobolev functions}

\author[R. Carlone]{Raffaele Carlone}
\address{Universit\`{a} ``Federico II'' di Napoli, Dipartimento di Matematica e Applicazioni ``R. Caccioppoli'', MSA, via Cinthia, I-80126, Napoli, Italy.}
\email{raffaele.carlone@unina.it}

\author[A. Fiorenza]{Alberto Fiorenza}
\address{Universit\`{a} ``Federico II'' di Napoli, Dipartimento di Architettura, via Monteoliveto, 3, I-80134, Napoli, Italy.}
\email{fiorenza@unina.it}

\author[L. Tentarelli]{Lorenzo Tentarelli}
\address{Universit\`{a} ``Federico II'' di Napoli, Dipartimento di Matematica e Applicazioni ``R. Caccioppoli'', MSA, via Cinthia, I-80126, Napoli, Italy.}
\email{lorenzo.tentarelli@unina.it}

\date{\today}
\begin{abstract}
For kernels $\nu$ which are positive and integrable we show that the operator $g\mapsto J_\nu g=\int_0^x \nu(x-s)g(s)ds$ on a finite time interval enjoys a regularizing effect when applied to H\"older continuous and Lebesgue functions and a ``contractive'' effect when applied to Sobolev functions. For H\"older continuous functions, we establish that the improvement of the regularity of the modulus of continuity is given by the integral of the kernel, namely by the factor $N(x)=\int_0^x \nu(s)ds$. For functions in Lebesgue spaces, we prove that an improvement always exists, and it can be expressed in terms of Orlicz integrability. Finally, for functions in Sobolev spaces, we show that the operator $J_\nu$ ``shrinks'' the norm of the argument by a factor that, as in the H\"older case, depends on the function $N$ (whereas no regularization result can be obtained).
\par
These results can be applied, for instance, to Abel kernels and to the Volterra function $\I(x) = \mu(x,0,-1) = \int_{0}^{\infty}x^{s-1}/\Gamma(s)\,ds$, the latter being relevant for instance in the analysis of the Schr\"odinger equation with concentrated nonlinearities in $\mathbb{R}^{2}$.
\end{abstract}


\maketitle
\begin{footnotesize}
 
 \end{footnotesize}


\section{Introduction}

\noindent Many mathematical models of physical phenomena deal with systems of Volterra integral equations with singular kernels (e.g. \cite{gv,hr,LZ}). In this paper, motivated by some nonlinear Volterra integral equations arising in Quantum Mechanics, we investigate the properties of convolution operators with kernels possibly more singular than the more known Abel ones. Namely, given a generic positive, locally integrable function $\nu$, we study the action of the operator $g\mapsto J_\nu g$ defined by
\begin{equation}
 \label{J_nu}
 (J_\nu g)(x):=\int_{0}^{x}\nu(x-s)g(s)ds,\quad x\geq0,
\end{equation}
on intervals $[0,T]$, with $T>0$ (this assumption being understood in the whole paper).

Precisely, we prove its regularizing effect in H\"older and Lebesgue spaces and its ``contractive'' effect in Sobolev spaces (where with ``contractive'' we mean that the Sobolev norm of $J_\nu g$ on $[0,T]$ can be estimated by the norm of $g$ times a constant that gets smaller as $T\to0$).

It is also worth highlighting that the assumption of local integrability of $\nu$ is the minimum requirement so that definition \eqref{J_nu} make sense in general. In fact, the aim of the paper (even though some results will require additional hypothesis) is to work with the least set of assumptions that are necessary in order to detect remarkable effects from the application of the operator $J_\nu$.

\medskip
A particular relevance in applications is acquired by the case
\begin{equation}\label{eq:i}
\nu(x)=\I(x): = \int_{0}^{\infty}\frac{x^{s-1}}{\Gamma(s)}ds
\end{equation}
(\textsc{Figure} \ref{fig:nu}), where the operator $J_\nu$ reads
\begin{equation}
 \label{eq:op_I}
 (J_\nu g)(x)=(I g)(x):=\int_{0}^{x}\I(x-s)g(s)ds.
\end{equation}
We observe that, if we denote by $\mu(x,\sigma,\alpha)$ the \emph{Volterra functions} defined by
\[
\mu (x,\sigma,\alpha):=\int_{0}^{\infty}\frac{x^{\alpha+s}s^\sigma}{\Gamma(\alpha+s+1)\Gamma(\sigma+1)}ds,
\]
then $\I(x)$ coincides with $\mu(x,0,-1)$, which is the so-called Volterra function \emph{of order} $-1$ (see \cite{E}, Section 18.3), that is discussed in Section \ref{sec-volt}.

In addition, recalling that a kernel $m\in L^{1}(0,T)$ is said a \textit{Sonine kernel} if it is a divisor of the unit with respect to the convolution operation, that is, if there exists another kernel $\ell\in L^{1}(0,T)$ such that
\[
\int_{0}^{x}\ell(x-s) m(s)ds=1,\quad\text{for a.e.}\: x\in[0,T],
\]
then, one can prove that $\I$ is a Sonine kernel, with $\ell(x)=-\gamma-\log x$, $\gamma$ representing the Euler-Mascheroni constant (see eq. \eqref{eq:Sonine}). The class of Sonine kernels is wide and there are many papers (see e.g. \cite{t} and references therein), starting with the pioneering one by Sonine (\cite{S}), where embedding theorems for integral operators with kernels displaying singularities at the origin of the type
\[
 a(x)x^{\alpha-1}\log^{m}\left(\frac{2T}{x}\right),\quad0<\alpha<1,\quad -\infty<m<\infty,
\]
are discussed. However, we stress that the results proved in the present paper are more general since they take into account also kernels that are more singular in a neighborhood of the origin, such as, indeed, the Volterra function $\I$, whose asymptotyc expansion near $0$ is given by $\frac{1}{x\log^2(1/x)}$ (see \eqref{eq:Iasympt}).

It is also worth mentioning that a first discussion on the operator $I$ is present in \cite{SKM}, whereas similar integral operators, but with more regular kernels, have been investigated more recently by \cite{cs,sc}. More in detail, in \cite{cs} it is analyzed the case of a certain class of \emph{almost decreasing} Sonine kernels in terms of \emph{weighted generalized} H\"older spaces, while in \cite{sc} an ``inverse'' operator is discussed within the framework of $L^{p}$ spaces. We also recall that in \cite{GM} some relevant features of Volterra functions are pointed out, such as asymptotic expansions and some striking relations with the Ramanujan integrals.

\medskip
The interest of the operator $I$ is mainly due to its applications in Quantum Mechanics, and precisely in the study of the Schr\"odinger equation with nonlinear point interactions in $\R^2$.

We recall briefly that a Schr\"odinger equation with a \emph{linear} point interaction with strength $\alpha\in\R$, placed at $\yv\in\R^2$, is
\[
 i \partial_t \psi(t) =  \ham \psi(t),
\]
where $\ham$ is a differential operator with domain
\[
 \begin{array}{l}
 \displaystyle \dom(\ham) = \bigg\{ \psi \in L^2(\R^2) \: :\: \forall \lambda>0,\: \psi = \phi_\lambda + q\, \mathcal{G}_{2,\lambda}(\cdot - \yv),\: \phi_\lambda \in H^2(\R^2),\:q\in\mathbb{C},\\[.4cm]
 \displaystyle \hspace{7cm} \lim_{\xv \to \yv} \phi_\lambda(\xv) = \left(\alpha+\tfrac{1}{2\pi} \log \tfrac{\sqrt{\lambda}}{2} + \tfrac{\gamma}{2\pi}\right)q\bigg\}
 \end{array}
\]
($\mathcal{G}_{2,\lambda}$ denoting the Green's function of $(-\Delta+\lambda)$ in $\R^2$) and action
\[
 (\ham + \lambda) \psi  = (- \Delta + \lambda) \phi_\lambda,\quad\forall \psi\in \dom(\ham).
\]
For a complete discussion on the solution of this equation through the theory of self-adjoint extension, we refer the reader to \cite{albe}. In addition, it is well known that, given an initial datum $\psi_0\in\dom(\ham)$, the solution of the associated Cauchy problem reads
\[
\psi(t,\xv)=(U_{0}(t)\psi_{0})(\xv)+\frac{i}{2\pi}\int_{0}^{t}U_{0}(t-s;\xv-\yv)\,q(s)ds,
\]
where $U_{0}(t)$ is the the \emph{propagator} of the free Schr\"odinger equation in $\R^2$ (with integral kernel $U_0(t; |\xv|) = e^{-\frac{|\xv|^{2}}{4 i t}}/2 i t$) and $q(t)$ (with a little abuse of notation that is usual in the literature) is a complex-valued function satisfying the so-called \emph{charge} equation
\begin{equation}
 \label{eq:lin_point}
 q(t)+\int_0^t\I(t-s)\left(4\pi\alpha-\log 4 + 2\gamma -\tfrac{i\pi}{2}\right)q(s)ds= 4\pi\int_0^t\I(t-s)(U_0(s)\psi_0)(\yv)ds.
\end{equation}
Now, a \emph{nonlinear} point interaction arises when one assumes that the strength of the interaction depends in fact on the function $q(t)$, and in particular when one sets $\alpha=\alpha_0|q(t)|^{2\sigma}$ ($\alpha_0\in\R\backslash\{0\}$, $\sigma>0$) in \eqref{eq:lin_point}, thus obtaining
\begin{equation}
 \label{eq:nonlin_point}
 q(t)+\int_0^t\I(t-s)\left(4\pi\alpha_0|q(t)|^{2\sigma}-\log 4 + 2\gamma -\tfrac{i\pi}{2}\right)q(s)ds= 4\pi\int_0^t\I(t-s)(U_0(s)\psi_0)(\yv)ds,
\end{equation}
(see \cite{CCF,CCT}). Since in the nonlinear case no theory of self-adjoint extensions is available, the relevance of the operator $I$ is clear: the well-posedness of the associated Cauchy problem is strictly related to the study of the existence and uniqueness of solutions of \eqref{eq:nonlin_point}, which strongly depends on the properties of $I$.

\begin{rem}
 Even though the application presented above concerns complex-valued functions, this papers only manages real-valued functions. However, one can check that the results of Section \ref{sob!} (which are actually required in \cite{CCF,CCT}) can be easily generalized to complex-valued functions.
\end{rem}

\medskip
Another topical example of integral kernels that are included in our general framework are the well known \emph{Abel kernels}, which correspond to the choice
\begin{equation}
 \label{eq-abel}
 \nu(x)=\frac{x^{\alpha-1}}{\Gamma(\alpha)},\quad 0<\alpha<1,
\end{equation}
in \eqref{J_nu}. These ones are very important in the theory of \emph{fractional integration} and \emph{generalized differentiation} (\cite{gv,SKM}) and, again, in Quantum Mechanics. In the study of nonlinear point interactions in $\R$ and $\R^3$, indeed, the resulting integral equations present the kernel \eqref{eq-abel}, with $\alpha=1/2$, in place of $\I$ (see \cite{ADFT,ADFT1,AT,CFNT1,CFNT}).

\medskip
Finally, we describe briefly the main results of the paper. They concern, as we told at the beginning, the properties of the operator $J_\nu$  in H\"older spaces, $L^{p}$ spaces and Sobolev spaces.

Preliminarily, since it is crucial in the following, we define the integral function of the kernel $\nu$
\begin{equation}
 \label{eq:int_funct_nu}
 N(x)=\int_0^x\nu(s)ds,\quad x>0.
\end{equation}
Since $\nu$ is always supposed positive and locally integrable, it turns out that $N$ is a positive, increasing and absolutely continuous function with $N(x)\to 0$ as $x\to0$. 

In the case of H\"older spaces, it is well known (\cite{gv}, Theorem 4.2.1 p. 70) that when the kernel is $\nu(t)=t^{\alpha-1}$, $0<\alpha<1$, the operator $J_\nu$ transforms $C^{0,\beta}$ functions into $C^{0,\alpha+\beta}$ ones, improving this way the regularity of the modulus of continuity. As a consequence of our main result of Section \ref{sectionhold} (Theorem \ref{casoholder}), we will see that more generally the improvement is at least given by the integral function of the kernel: the phenomenon that the power $\alpha-1$ gives as improvement the exponent $\alpha$ is therefore true also for any locally integrable kernel which is assumed just equivalent to a decreasing function in a neighborhood of the origin and not blowing too much (derivative bounded above, for instance) in its domain.

In the case of $L^{p}$ spaces, it is well known (\cite{gv}, Theorem 4.1.4 p. 67) that when the kernel is $\nu(t)=t^{\alpha-1}$, the operator $J_\nu$ transforms $L^p$ functions, $1<p<1/\alpha$, into $L^{p/(1-\alpha p)}$ functions. To a minor integrability of the kernel corresponds a minor gain of integrability for $J_{\nu}$, and apparently the gain disappears when the kernel is just $L^1$. As a consequence of our main result of Section \ref{sectionellepi} (Theorem \ref{casoLp}), we will show that \it any \rm kernel locally integrable (again, we assume that it is equivalent to a decreasing function in a neighborhood of the origin) gives an improvement of integrability, measured in terms of Orlicz spaces. The improvement is strictly linked to the Orlicz integrability of the kernel, hence it always exists: a classical, remarkable theorem in Orlicz spaces theory (see e.g. \cite{krasnorut}, p. 60) tells that any function $L^1$ is always in some Orlicz space strictly contained in $L^1$. Furthermore, in the case $p=\infty$, we show (Proposition \ref{pro:infinity}), under the unique assumption of local integrability, that $J_\nu$ transforms $L^\infty$ functions in continuous functions and that the $L^\infty$ norm of $J_\nu g$ on $[0,T]$ is controlled by the norm of $g$ times $N(T)$.

Finally, in the case of Sobolev spaces, it is well known (\cite{gv}, Theorem 4.2.2 p. 73) that when the kernel is $\nu(t)=t^{\alpha-1}$, the operator $J_\nu$ transforms $W^{\theta,1}$ functions, with $0<\theta<1-\alpha$, in $W^{\theta+\alpha-\varepsilon,1}$ functions. Analogous results for $W^{\theta,p}$ functions are discussed in \cite{AT,k}. In this case, the minor integrability of the kernel yields a minor gain in the Sobolev index, which disappears when the kernel is just $L^1$ (also the preservation of the index is not straightforward). As a consequence of our main result of Section \ref{sob!} (Theorem \ref{contr_lemma}), we will show that when $p=2$, provided $\theta\neq1/2$, the Sobolev index is in fact preserved and  the Sobolev norm of $J_\nu g$ is bounded, up to a multiplicative constant, by the norm of $g$ times $N(T)$. Furthermore, we will prove that (almost) the same result holds for $H^{1/2}$ functions, but just in the case $\nu=\I$ (Theorem \ref{contr_lemma_2}), and for $W^{1,1}$ functions (Theorem \ref{contr_lemma_3}).


\section{The Volterra kernel $ \I $}
\label{sec-volt}

\noindent
Since the case of a kernel equal to the Volterra function $\I$ (defined by \eqref{eq:i}, \textsc{Figure} \ref{fig:nu}) is the most relevant in the applications, it is worth stressing some basic features of $\I$. In this way one can easily see that the abstract results established in the following sections can be actually applied to this kernel.

First, we recall (see \cite{E,SKM}) that $\I$ is analytic for $t>0$ and that
\begin{equation}
\label{eq:Iasympt}
\begin{array}{ll}
 \displaystyle \I(x) = \frac{1}{x \log^2 \left(\frac{1}{x}\right)}\left[1 + \mathcal{O}(\left|\log x \right|^{-1}) \right], & \quad\text{as }x\to 0\\[1cm]
 \displaystyle \I(x) = e^{x}+\mathcal{O}(x^{-1} ),                                                                          & \quad\text{as }x\to +\infty.
\end{array}
\end{equation}
Consequently, the first expansion shows that $ \I\in L^{1}_{\textrm{loc}}(\R^+)$ and that $\I\not\in L^{p}_{\textrm{loc}}(\R^+)$, for any $p>1$.

\begin{figure}[ht]
    \centering
    \includegraphics[width=0.4\textwidth]{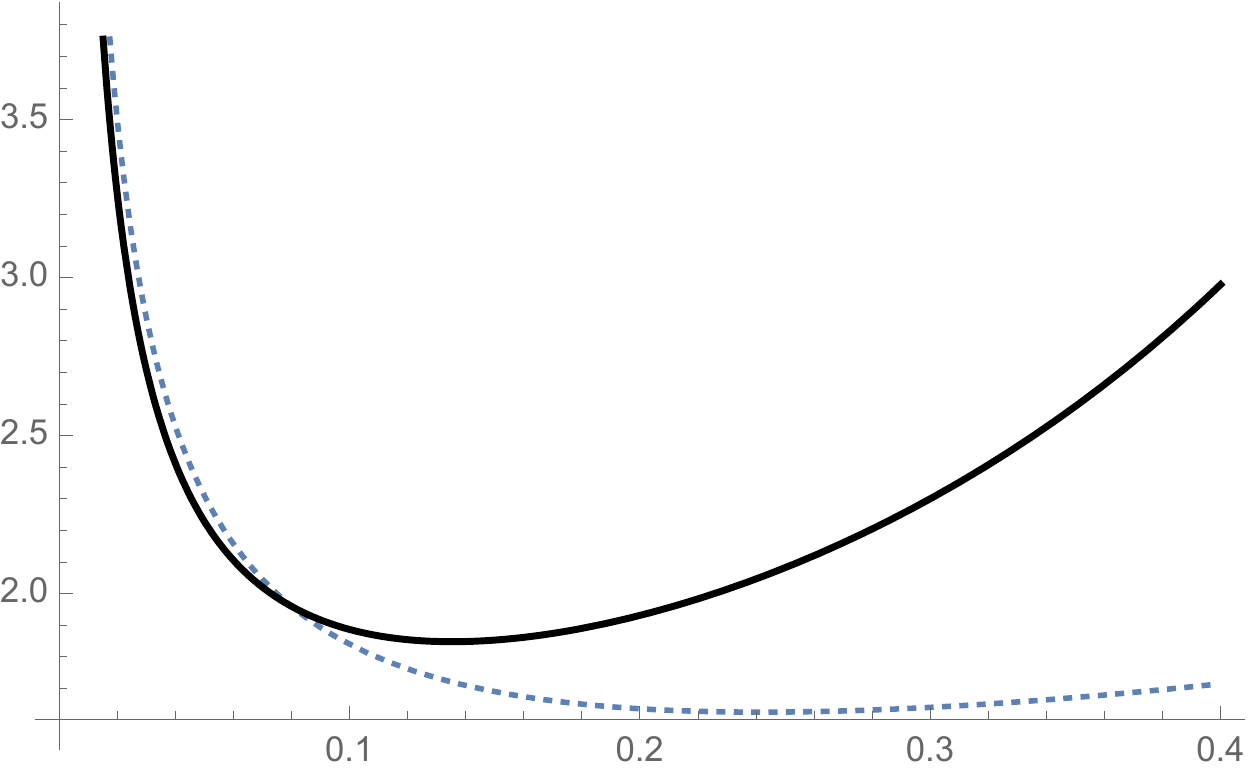}
    \caption{The plot of $\I(x)$ is in black,  the plot of the first order of the asymptotic expansion of $\I(x)$ around 0 is dotted.}
    \label{fig:nu}
\end{figure}

Also the derivative and the integral function of $\I$ will play a crucial role in the sequel. Hence, recalling (see again \cite{E,SKM})
\begin{equation}
 \label{der_volt}
 \frac{d}{dx}\mu(x,0,n)=\mu(x,0,n-1),\quad n\in\Z,\quad n\leq0,
\end{equation}
and $\I(x):=\mu(x,0,-1)$, we stress that
\begin{equation}
\label{derivat}
\frac{d}{dx}\I(x) = \mu (x,0,-2) \underset{x \to 0}{=} \frac{1}{x^2 \log^2 \left(\frac{1}{x}\right)}\left[-1 + \mathcal{O}(\left|\log x \right|^{-1}) \right]
\end{equation}
and that
\begin{equation}
\label{primit}
\frac{d}{dx}\mu(x,0,0) = \I(x).
\end{equation}
Furthermore, we can state the following lemma.

\begin{lem}
 \label{lem:conv}
 The function $\I$ is convex on $\R^+$ and  admits a positive minimum.
\end{lem}

\begin{proof}
 The second part is immediate since $\I$ is continuous, positive and coercive by \eqref{eq:Iasympt}. On the other hand, in order to prove the second part, it is sufficient to show $\frac{d^2}{dx^2}\I(x)\geq0$; namely, by \eqref{der_volt}, that $\frac{d^3}{dx^3}\mu(x,0,0)\geq0$. Now, following \cite{GM} (eq. (3.1)) and \cite{h}, we find that
 \[
  \mu(x,0,0)=e^x-R(x), \quad \mbox{where} \quad R(x):=\int_0^\infty\frac{e^{-s x}}{s(\log^2s+\pi^2)}ds
 \]
 denotes the \emph{Ramanujan} function (\textsc{Figure} \ref{fig:erre}). Hence, $\frac{d^3}{dx^3}\mu(x,0,0)=e^x-\frac{d^3}{dx^3}R(x)$ and, since $R$ is \emph{completely monotonic} (i.e., for every $k>0$, $R^{(k)}$ does exist and $(-1)^k R^{(k)}\geq0$), this entails that $\frac{d^3}{dx^3}\mu(x,0,0)\geq0$ and thus that $\I$ is convex.
\end{proof}

\begin{figure}[ht]
    \centering
    \includegraphics[width=0.4\textwidth]{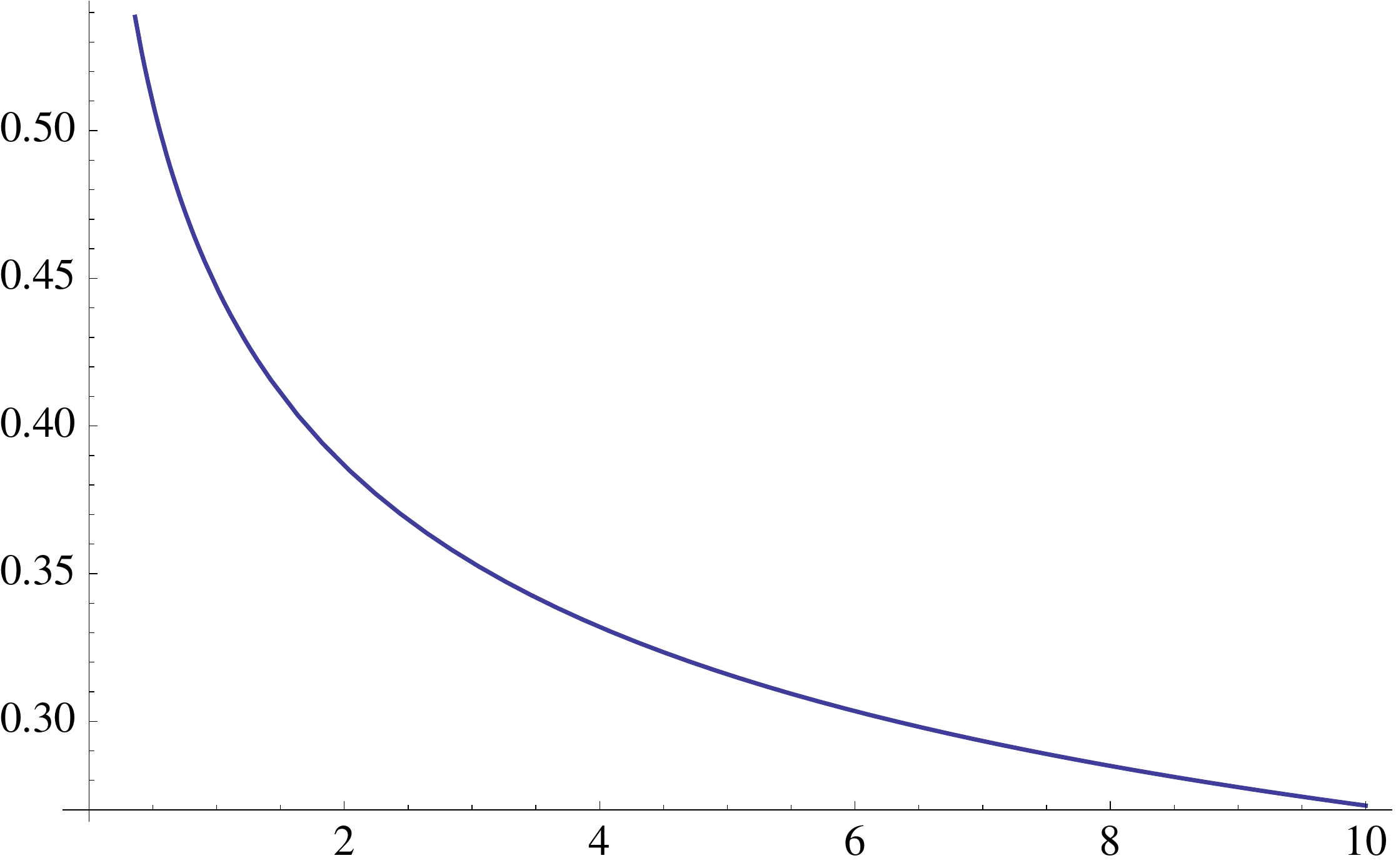}
    \caption{The plot of $R(x)$ .}
    \label{fig:erre}
\end{figure}

It is also convenient to introduce the function 
\begin{equation}\label{enne}
\mathcal{N}(x)=\int_{0}^{x}\I(s)ds,\quad x\geq0 ,
\end{equation}
(\textsc{Figure} \ref{fig:enne}). By the properties of $ \I$, we see that $\NN$ is positive, increasing and absolutely continuous on bounded intervals intervals. Moreover, $\NN(x)\to0$ as $x\to0$, and precisely
\begin{equation}
 \label{eq:Nasympt}
 \NN(x) = \frac{1}{\log\left(\tfrac{1}{x}\right)}+\mathcal{O}(|\log x|^{-2}),\quad \mbox{as}\: x\to0.
\end{equation}

\begin{figure}[ht]
    \centering
    \includegraphics[width=0.5\textwidth]{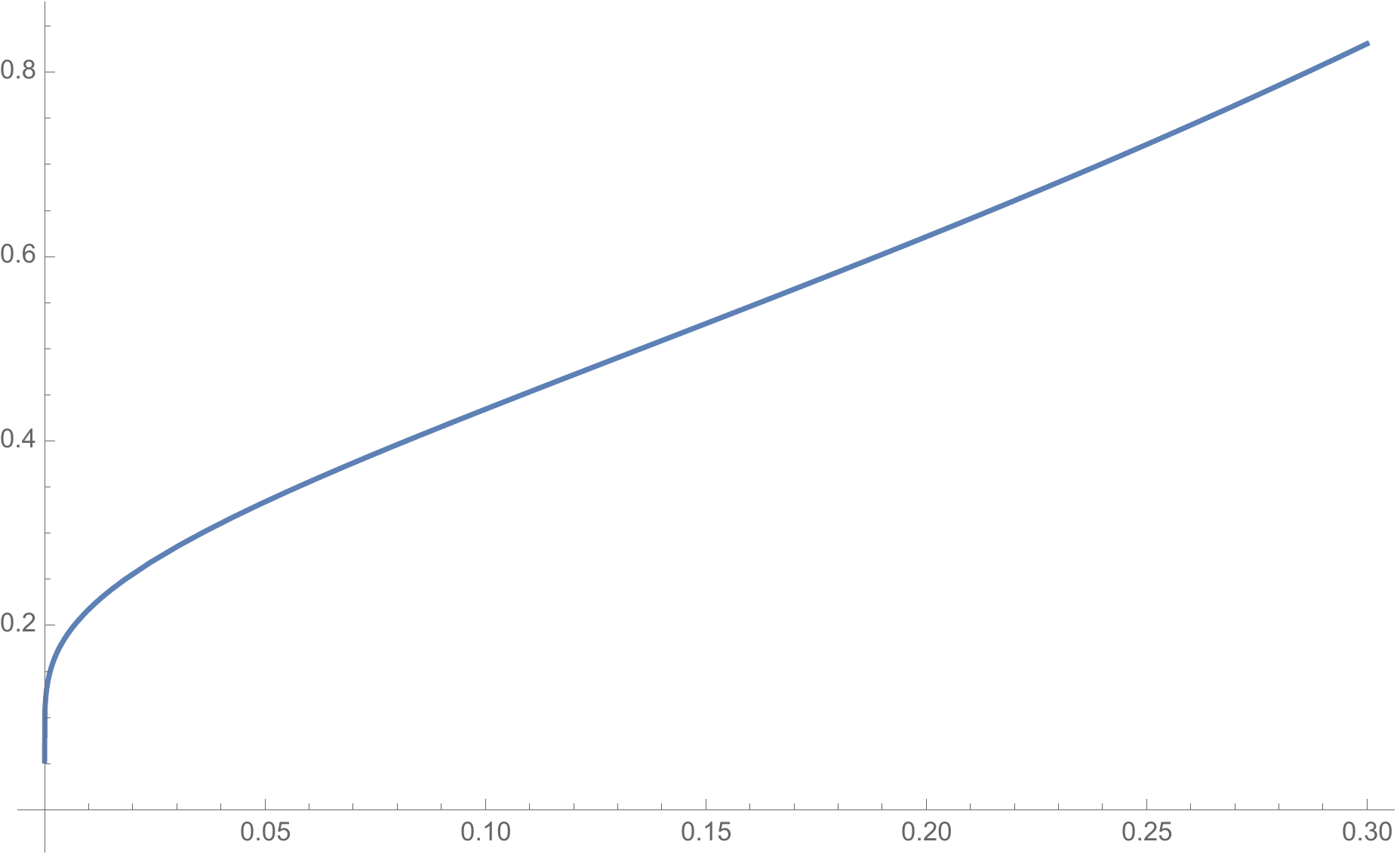}
    \caption{Plot of $\mathcal{N}(x)$ around zero}
    \label{fig:enne}
\end{figure}

\begin{rem}
 One easily sees that, as one sets $\nu=\I$ in \eqref{eq:int_funct_nu}, $N$ is equal to $\NN$. On the other hand, from \eqref{der_volt} one also notes that $\NN$ coincides, up to an additive constant, with $\mu(\cdot,0,0)$.
\end{rem}

Finally, we point out a relevant property of the operator $I$ defined by \eqref{eq:op_I}, which is strictly connected to the fact that $\I$ is a Sonine kernel. First, define the integral operator
\[
 (\Phi g)(x) : = \int_0^x \phi(x - s) g(s)ds,\quad \text{where}\quad \phi(x) = - \gamma - \log x.
\]
Then, one notes that, as $\phi\in L^1(0,T)$, $\Phi$ is well defined for each function $g\in L^1(0,T)$. In addition, one can prove the following result.

\begin{pro}
 \label{pro:i_identity}
 If $ g \in L^1(0,T) $, then
 \[
  \big( \Phi (I g) \big)(x) = \big( I (\Phi g) \big)(x) = \int_0^x g(s)ds,\quad \forall x\in[0,T].
 \]
\end{pro}

\begin{proof}
 We first observe that one has
 \begin{equation}
  \label{eq:i_identity}
  \int_{0}^{x} \I(x - s) \phi(s)ds = 1.
 \end{equation}
 In \cite{SKM}, Lemma 32.1, it is indeed claimed that (setting $ \alpha = 1, h = 0 $ therein)
 \begin{equation}
  \label{eq:Sonine}
  \int_{0}^{x} \left(\log s - \psi(1) \right) \frac{d}{dx}\mu(x - s,0,0) ds= -1,
 \end{equation}
 with $\psi(1)=-\gamma$, and since from \eqref{primit} $ \frac{d}{dx}\mu(x,0,0) = \I(x) $, \eqref{eq:i_identity} is proved.

 Now, in the expression
 \[
  \big( I (\Phi g) \big)(x) = \int_0^{x} \int_0^{x-s} \I(s) \phi(x-s-\sigma) g(\sigma)d\sigma\,ds,
 \]
 we exchange the order of the integration, since
 \[
  \begin{array}{l}
   \displaystyle \int_0^{x} \int_0^{x-s} \I(s) \phi(x-\sigma-s) g(\sigma) d\sigma\,ds  + \int_0^{x} \int_0^{x-\sigma} \I(s) \phi(x-\sigma-s) g(\sigma)ds\,d\sigma= \\[.7cm]
   \displaystyle = \int_0^{x} \int_0^{x} \I(s) \phi(|x-\sigma-s|) g(\sigma)ds\,d\sigma = 2  \int_0^{x} \int_0^{x-s} \I(s) \phi(|x-\sigma-s|) g(\sigma)d\sigma\,ds,
  \end{array}
 \]
and using \eqref{eq:i_identity}, we conclude that
 \[
  \big( I (\Phi g) \big)(x) = \int_0^{x} \int_0^{x-\sigma} \I(s) \phi(x-\sigma-s) g(\sigma)ds\,d\sigma = \int_0^{x} g(\sigma)d\sigma.
 \]
 Finally, given that an easy change of variable shows $\big( \Phi (I g) \big)(x) = \big( I (\Phi g) \big)(x)$, the proof is complete.
\end{proof}

\begin{rem}
 We note that, in view of Theorem \ref{contr_lemma_3}, it is
 \[
  \frac{d}{dx}\big( \Phi (I g) \big)(x) = \frac{d}{dx}\big( I (\Phi g) \big)(x) = g(x).
 \]

\end{rem}


\section{Regularization in H\"older spaces}\label{sectionhold}

\noindent
Let $0<\tau_0<\infty$ and let $\nu$, $\tilde\nu$ be absolutely continuous and positive in $]0,\tau_0]$. We say 
that $\nu$, $\tilde\nu$ are $equivalent$ if there exist two positive constants $c_1,c_2$ such that
$$ c_1\nu(x)\le \tilde\nu(x) \le c_2\nu(x)\quad \forall x\in ]0,\tau_0].$$
Of course any function equivalent to $\nu$ in $]0,\tau_0]$ is of type $b\nu$, where $b=b(\cdot)$ is absolutely continuous in $]0,\tau_0]$ and such that
\begin{equation}\label{bi} 
0<b^-\le b(x)\le b^+<\infty \quad 
\forall x\in ]0,\tau_0].
\end{equation}

The statements of this section hold for certain functions $\nu$ which are decreasing in intervals of the type $(0,\tau_{0})$ and, more generally, they hold for functions equivalent to decreasing functions in intervals of the type $(0,\tau_{0})$. For the sake of simplicity, the functions $\nu$ in the statements will be always assumed equivalent to decreasing functions in intervals of the type $(0,\tau_{0})$, and the corresponding decreasing functions will be written as products $b\nu$, where $b=b(\cdot)$ is an absolutely continuous function in $]0,\tau_0]$ satisfying \eqref{bi}.

\begin{lem}\label{primolemma}
Let $0<\beta<1$, and let $\nu\in AC(]0,T])\cap L^1(0,T)$ be positive and equivalent to a decreasing function in $(0,\tau_{0})$ for some $0<\tau_{0}\leq T$. If
\begin{equation}\label{decresenzaderiv}
x\to\frac{1}{x^{\beta}}\int_0^x b(s)\nu(s)ds
\,\,\searrow\,\, \textrm{in}\,\, (0, \tau_{0}),
\end{equation}
\noindent 
then
\begin{equation}\label{tesilemma}
\frac{\nu(\varepsilon\tau)\varepsilon}{N(\varepsilon)}\leq 
\textit{c}\left(b(\cdot),\nu(\cdot),T\right)\tau^{\beta-1}\quad\forall\,\, 1<\tau<\frac{T}{\varepsilon},\,\forall\,\, 0<\varepsilon<T,
\end{equation}
where $N$ is defined by \eqref{eq:int_funct_nu}.
\end{lem}

\begin{rem}
Inequalities coming from assumptions of monotonicity of ratios between functions and powers are very well known among researchers working in Orlicz spaces. Some proofs of such inequalities work also without the assumption of convexity (the reader may compare this lemma e.g. with Theorem 3 in \cite{raoren} or with the results in Section 3 of \cite{Maligr}), however, the main feature of \eqref{tesilemma} is that it has been obtained from assumptions of monotonicity which hold only in a neighborhood of the origin and not in the whole domain of the functions involved (where, however, at least a boundedness is required; this assumption appears implicitly in the hypothesis of continuity which holds until the endpoint $T$). 
\end{rem}

\begin{proof}
We preliminarly note that, since $b\nu$ is decreasing in $(0,\tau_{0})$, 
\begin{equation}\label{decreconderiv}
\displaystyle \nu(x)\le \frac{1}{b^-}b(x)\nu(x)\leqslant
\frac{1}{b^-}\cdot \frac{1}{x}\int_0^x b(s)\nu(s)ds
\le
\frac{b^+}{b^-}\frac{N(x)}{x} \,\,\,\, \forall x\in (0, \tau_{0}]\, .
\end{equation}

\noindent Let $0<\varepsilon<T$, $\displaystyle 1<\tau<\frac{T}{\varepsilon}$. If $\varepsilon\tau\leqslant\tau_{0}$, by \eqref{decreconderiv} and \eqref{decresenzaderiv} respectively,
\begin{equation}\label{prima}
\frac{\nu(\varepsilon\tau)\,\varepsilon}{N(\varepsilon)}\leqslant
\frac{b^+}{b^-}\frac{N(\varepsilon\tau)}{\varepsilon\,\tau}\frac{\varepsilon}{N(\varepsilon)}
=\frac{b^+}{b^-}\frac{N(\varepsilon\tau)}{(\varepsilon\,\tau)^{\beta}}\frac{(\varepsilon\,\tau)^{\beta}}{\tau\,N(\varepsilon)}
\leqslant\left(\frac{b^+}{b^-}\right)^2
\frac{N(\varepsilon)}{\varepsilon^{\beta}}\frac{(\varepsilon\tau)^{\beta}}{\tau\,N(\varepsilon)}=\left(\frac{b^+}{b^-}\right)^2\tau^{\beta-1}
\end{equation}

\noindent and, if $\tau_{0}<\varepsilon\tau<T$, $\displaystyle 0<\varepsilon\leqslant\tau_{0}$ then again by \eqref{decresenzaderiv}, $\displaystyle \frac{N(\varepsilon)}{\varepsilon^{\beta}}\geqslant \frac{b^-}{b^+} \frac{N(\tau_{0})}{\tau_{0}^{\beta}}$ and therefore, setting
$
\nu^+=\displaystyle\max_{[\tau_{0},T]}\nu\, ,
$
\begin{equation}\label{seconda}
\frac{\nu(\varepsilon\tau)\,\varepsilon}{N(\varepsilon)}\leqslant
\nu^+\varepsilon^{1-\beta}\frac{\varepsilon^{\beta}}{N(\varepsilon)}\leqslant\frac{b^+}{b^-}
\nu^+\left(\frac{T}{\tau}\right)^{1-\beta}\frac{\tau_{0}^{\beta}}{N(\tau_{0})}=\frac{b^+}{b^-}\frac{\nu^+T^{1-\beta}\tau_{0}^{\beta}}{N(\tau_{0})}\tau^{\beta-1}
\le\frac{b^+}{b^-} \frac{\nu^+T}{N(\tau_{0})}\tau^{\beta-1}
\end{equation}

\noindent Finally, if $\tau_{0}<\varepsilon\tau<T$, $\displaystyle \varepsilon>\tau_{0}$ then, since $\nu>0$ implies that $N$ is increasing, 
\begin{equation}\label{terza}
\frac{\nu(\varepsilon\tau)\,\varepsilon}{N(\varepsilon)}\leqslant
\nu^+\varepsilon^{1-\beta}\frac{\varepsilon^{\beta}}{N(\varepsilon)}\leqslant
\nu^+\left(\frac{T}{\tau}\right)^{1-\beta}\frac{(\varepsilon\,\tau)^{\beta}}{N(\tau_{0})}<\frac{\nu^+T^{1-\beta}T^{\beta}}{\tau^{1-\beta}\,N(\tau_{0})}=\frac{\nu^+\,T}{N(\tau_{0})}\tau^{\beta-1}.
\end{equation}
From \eqref{prima}, \eqref{seconda}, \eqref{terza}, we get \eqref{tesilemma}.
\end{proof}

The following statement is an immediate consequence of Lemma \ref{primolemma}.

\begin{cor}\label{usatodavvero}
In the same assumptions of Lemma \ref{primolemma}, for any $\alpha<1-\beta$ it is
\[
\int_{1}^{\frac{T}{\varepsilon}}\,(\tau+1)^{\alpha-1}\frac{\nu(\varepsilon\tau)\varepsilon}{N(\varepsilon)}d\tau\le\textit{c}\left(b(\cdot),\nu(\cdot),T\right)\int_{1}^{\frac{T}{\varepsilon}}\,(\tau+1)^{\alpha-1}\tau^{\beta-1}d\tau\le 
\textit{c}\left(b(\cdot),\nu(\cdot),T,\alpha+\beta\right)<
\infty
\]
uniformly in $\varepsilon$, $0<\varepsilon<T$.
\end{cor}

The next lemma is trivially true for decreasing functions $\nu$ (see the CASE (i) of the proof), and it provides a version of the inequality in case of functions which are decreasing only in a neighborhood of the origin (however, as in the remark above, one can see that, again, an assumption of boundedness has been made implicitly).

\begin{lem}\label{secondolemma}
If $\nu\in AC(]0,T])\cap L^1(0,T)$ is positive and equivalent to a decreasing function in $(0,\tau_{0})$ for some $0<\tau_{0}\leq T$,
then
\begin{equation}\label{thesissecondolemma}
\int_x^y\nu (s)ds\le \textit{c}\left(b(\cdot),\nu(\cdot),T\right)\int_0^{y-x}\nu(s)ds
\quad \forall x,y\in (0,T), \frac y2\le x<y.
\end{equation}
\end{lem}

\begin{proof}
We examine the three cases

\noindent
(i) $x<y\le \tau_0$

\noindent
(ii) $\tau_0\le x < y \le T$

\noindent
(iii) $x<\tau_0<y\le T$

\medskip

\noindent
CASE (i): Since $b\nu$ is decreasing in $(0, \tau_0)$,
$$
\int_x^y \nu(s)ds=\int_0^{y-x}\nu(s+x)ds
\le \frac{b^+}{b^-}\int_0^{y-x}\nu(s)ds
\qquad \forall x,y\in (0,\tau_0), \frac y2\le x<y.
$$

\noindent
CASE (ii): We have
\begin{equation}\label{secondstep}
\int_x^y \nu(s)ds\le \nu^+(y-x)=\frac{\nu^+}{\int_0^{y-x}\nu(s)ds}(y-x)\int_0^{y-x}\nu(s)ds
\qquad \forall x,y\in (\tau_0, T), \frac y2\le x<y,
\end{equation}
where $\nu^+=\displaystyle\max_{[\tau_{0},T]}\nu$.

There are two possibilities:

\noindent
$\textrm{(ii)}_1$ $y-x\le \tau_0$

\noindent
$\textrm{(ii)}_2$ $y-x > \tau_0$

In the case $\textrm{(ii)}_1$, since $b\nu$ is decreasing in $(0, \tau_0)$,
$$
\int_0^{y-x}\nu(s)ds\ge \frac{b^-}{b^+}\nu(\tau_0)(y-x) 
$$
and therefore from \eqref{secondstep}
$$
\int_x^y \nu(s)ds\le \frac{b^+}{b^-} \frac{\nu^+}{\nu(\tau_0)}\int_0^{y-x}\nu(s)ds
\qquad \forall x,y\in (\tau_0, T), \frac y2\le x<y;
$$ 
in the case $\textrm{(ii)}_2$, setting $\nu^-=\displaystyle\min_{]0,T]}\nu>0$,
$$
\int_0^{y-x}\nu(s)ds\ge \tau_0\nu^- 
$$
and then, using $y-x\le y/2\le T/2$, from \eqref{secondstep} we get
$$
\int_x^y \nu(s)ds\le \frac{\nu^+}{\tau_0\nu^-}\frac{T}{2}\int_0^{y-x}\nu(s)ds
\qquad \forall x,y\in (\tau_0, T), \frac y2\le x<y.
$$

\noindent
CASE (iii): We have
$$
\int_x^y \nu(s)ds=\int_x^{\tau_0} \nu(s)ds+\int_{\tau_0}^y \nu(s)ds
$$
and applying CASE (i) with $y$ replaced by $\tau_0$ and  CASE (ii) with $x$ replaced by $\tau_0$,
\[
\int_x^y \nu(s)ds\le \frac{b^+}{b^-}\int_0^{\tau_0-x} \nu(s)ds+
\max \left\{\frac{b^+}{b^-}\frac{\nu^+}{\nu(\tau_0)},\frac{\nu^+T}{2\tau_0\nu^-}\right\}\int_0^{y-\tau_0}\nu(s)ds.
\]
Since in our case $\tau_0<y$, it is $\tau_0-x<y-x$, hence the first term can be estimated by the right hand side of \eqref{thesissecondolemma}; similarly, since 
$x<\tau_0$, it is $y-\tau_0<y-x$ and the same conclusion holds for the second term.
\end{proof}

In next theorem we are going to consider an assumption on $\nu$ stronger (as we are going to see) with respect to that one 
of Lemma \ref{primolemma}: in the case $b\equiv 1$ (a similar digression can be done in the general case, replacing $\nu$ by $b\nu$) we will assume that the positive function $\nu\in AC(]0,T])\cap L^1(0,T)$ is such that
$x\to x^{1-\beta}\nu (x)$ is decreasing in $(0,\tau_0)$ for some $0<\tau_{0}\leq T$, $0<\beta<1-\alpha$, where $\alpha$ is a given number in $(0,1)$. It is easy to verify that this latter assumption implies that the function $N$, defined by \eqref{eq:int_funct_nu}, is such that 
\[
\frac{N(x)}{x^{\beta}}\searrow\,\, \textrm{in}\,\, (0, \tau_{0})\, ,
\]
and also that $\nu$ is decreasing in $(0,\tau_0)$ (because 
$\nu (x)=x^{\beta-1}\cdot x^{1-\beta}\nu (x)$ is product 
of positive decreasing functions).

Let us verify the first assertion.
Since $x\to x^{1-\beta}\nu(x)$ decreasing in $(0,\tau_0)$, it is (note that, since also
$\nu$ is absolutely continuous, their derivatives exist a.e.) 
$$
\sigma\frac{d}{d\sigma}\nu(\sigma)\leqslant (\beta-1)\nu(\sigma)\quad {\rm for}\,\,{\rm a.e.}\,\,\sigma\in (0,\tau_0)\, ,
$$ 
hence, integrating the above inequality in $(\varepsilon, x)$, where $0<\varepsilon<x<\tau_0$, and noting that $\sigma\nu(\sigma)$ is absolutely continuous too,
we have
$$
x\nu(x)-\varepsilon\nu(\varepsilon)=\int_\varepsilon^x\sigma\frac{d}{d\sigma}\nu(\sigma)d\sigma+
\int_\varepsilon^x\nu(\sigma)d\sigma$$
$$\le (\beta-1)
\int_\varepsilon^x\nu(\sigma)d\sigma+
\int_\varepsilon^x\nu(\sigma)d\sigma = \beta \int_\varepsilon^x\nu(\sigma)d\sigma\, .
$$
If we let $\varepsilon\to 0$, since $\nu$ is decreasing, it is
$$0<\varepsilon\nu(\varepsilon)\le\int_0^\varepsilon\nu(\sigma)d\sigma\to 0\, ,$$
and therefore we get
$$x\nu(x)\le \beta N(x)\,\quad \forall x\in (0,\tau_0) \, ,$$
from which the assertion follows.

On the other hand, the fact that the assumption is really stronger is shown by the following

\begin{example}
Let $1<\tau_0<\infty$, and let
$$
\nu(\sigma)=
\begin{cases}
\displaystyle\frac{1}{2\sqrt{\sigma}} \quad \textrm{if} \,\,\sigma\in (0,1)\\ \\
\displaystyle\frac{1}{2} \qquad\,\,\, \textrm{if} \,\,\sigma\in (1,\tau_0)
\end{cases}
$$
so that
$$
N(x)=
\begin{cases}
\sqrt{x} \quad\,\,\,\, \textrm{if} \,\,\,\, x\in (0,1)\\ \\
\displaystyle\frac{x+1}{2} \,\,\, \textrm{if} \,\,x\in (1,\tau_0)
\end{cases}
$$
Then 
$$
\frac{N(x)}{x^{\beta}}\searrow\,\, \textrm{in}\,\, (0,\tau_0)
$$
is satisfied for $\beta=\tau_0/(\tau_0+1)$, while for the same $\beta$ the function $x\to x^{1-\beta}\nu (x)$ is not decreasing in $(0,\tau_0)$ (because it is not decreasing in $(1,\tau_0)$).
\end{example}

Before the statement of the main theorem of this section, we observe that the kernels $\nu$ of our interest are such that their difference quotients are bounded above, i.e. there exists a constant $K>0$ such that
$$
\nu (y) - \nu(x) \le K (y-x) \qquad \forall x,y \in (0,T), x<y
$$
This property (which holds automatically, in particular, for all the kernels $\nu$ which are decreasing in the whole $(0,T)$) is expressed in an equivalent way in the assumption \eqref{diffquobdd} below.

\begin{teo}\label{casoholder}
If $g\in C^{0,\alpha}([0,T])$, $0<\alpha<1$, $0<T<\infty$, $g(0)=0$, and if 
$\nu\in AC(]0,T])\cap L^1(0,T)$, $\nu >0$, is such that
\begin{equation}\label{bastaquestaipotesi}
x\to x^{1-\beta}b(x)\nu (x) \,\,\textrm{is decreasing in} \,\,(0,\tau_0)
\end{equation}
$$
\textrm{for some}\,\, 0<\tau_{0}\leq T\, , 0<\beta<1-\alpha\, ,
\,\, b\in AC(]0,\tau_0[)\, ,\,
0<b^-\le b(x)\le b^+<\infty\, ,
$$
\begin{equation}\label{diffquobdd}
|\nu (x) - \nu(y)| \le  \nu (x) - \nu(y) + K_0 (y-x)\,\,\,\,\textrm{for some}\,\,K_0>0,\quad \forall x,y \in (0,T), x<y
\, ,
\end{equation}

then, setting
$$
J_\nu g(t)=\int_0^t \nu(t-s)g(s)ds\qquad t\in (0,T),
$$
it is 
\begin{equation}\label{thesis}
|J_\nu g(x)-J_\nu g(y)|\le c\left(b(\cdot),\nu(\cdot),T,\alpha,\beta\right) [g]_\alpha |x-y|^\alpha N(|x-y|)\,\, \forall x,y\in (0,T), \frac y2\le x<y\, ,
\end{equation}
where
$$
N(x)=\int_{0}^{x}\nu(\sigma)d\sigma\, .
$$
\end{teo}

\begin{proof}
Let $x,y\in (0,T), \frac y2\le x<y$. It is
\begin{equation}\nonumber
\begin{split}
&J_\nu g(y)-J_\nu g(x)=
\int_0^y \nu(y-s)g(s)ds-\int_0^x \nu(x-s)g(s)ds\\
&=\int_0^y \nu(s)g(y-s)ds-\int_{y-x}^y \nu(s-y+x)g(y-s)ds\\
&=\int_0^y \nu(s)g(y)ds-\int_0^y \nu(s)[g(y)-g(y-s)]ds\\
&-\int_{y-x}^y \nu(s-y+x)g(y)ds+\int_{y-x}^y \nu(s-y+x)[g(y)-g(y-s)]ds\\
&=g(y)\int_0^y \nu(s)ds-\int_0^{y-x} \nu(s)[g(y)-g(y-s)]ds-\int_{y-x}^y \nu(s)[g(y)-g(y-s)]ds\\
&-g(y)\int_{y-x}^y \nu(s-y+x)ds+\int_{y-x}^y \nu(s-y+x)[g(y)-g(y-s)]ds\\
&=g(y)[N(y)-N(x)]-\int_0^{y-x} \nu(s)[g(y)-g(y-s)]ds\\
&+\int_{y-x}^y [\nu(s-y+x)-\nu(s)][g(y)-g(y-s)]ds.
\end{split}
\end{equation}
Therefore
\[
\begin{split}
&|J_\nu g(y)-J_\nu g(x)|\\
&\le|g(y)[N(y)-N(x)]|+\left\vert\int_0^{y-x} \nu(s)[g(y)-g(y-s)]ds\right\vert\\
&+\left\vert\int_{y-x}^y [\nu(s-y+x)-\nu(s)][g(y)-g(y-s)]ds\right\vert:=A+B+C.
\end{split}
\]
We estimate each term in turn. Since $g\in C^{0,\alpha}([0,T])$, $g(0)=0$, and since $N$ is increasing and $0<\alpha<1$,
\[
\begin{split}
&A=|g(y)[N(y)-N(x)]|\\
&\le [g]_\alpha y^\alpha [N(y)-N(x)]\\
&= [g]_\alpha (y^\alpha - x^\alpha) [N(y)-N(x)]+[g]_\alpha x^\alpha [N(y)-N(x)]\\
&\le [g]_\alpha (y - x)^\alpha [N(y)-N(x)]+[g]_\alpha x^\alpha [N(y)-N(x)].
\end{split}
\]
By \eqref{bastaquestaipotesi} the function $b\nu$ is decreasing in $(0,\tau_0)$, hence by \eqref{thesissecondolemma}
$$N(y)-N(x)=\int_0^y \nu(s)ds-\int_0^x \nu(s)ds=\int_x^y \nu(s)ds$$
$$\le\textit{c}\left(b(\cdot),\nu(\cdot),T\right)\int_0^{y-x} \nu(s)ds=\textit{c}\left(b(\cdot),\nu(\cdot),T\right)N(y-x)$$
and therefore
$$A\le \textit{c}\left(b(\cdot),\nu(\cdot),T\right)[g]_\alpha (y - x)^\alpha N(y-x)+[g]_\alpha x^\alpha [N(y)-N(x)].$$
The first term can be estimated by the right hand side of \eqref{thesis}. As to the second term, we begin observing that
$$x^\alpha [N(y)-N(x)]=x^\alpha \frac{N(y)-N(x)}{y-x}(y-x)
= x^\alpha \nu(\xi)(y-x),$$ 
for some $\xi\in (x,y)$.

We consider first the case $\xi\le \tau_0$, so that
$$x^\alpha \nu(\xi)(y-x)=
b(\xi)^{-1}x^\alpha b(\xi)\nu(\xi)(y-x)
\le \frac{1}{b^-} x^\alpha b(x)\nu(x)(y-x)$$ 
(because $b\nu$ is decreasing in $(0,\tau_0)$).
 Since $\alpha<1-\beta$, by \eqref{bastaquestaipotesi}
$$
x\to x^{\alpha}b(x)\nu (x) \,\,\textrm{is decreasing in} \,\,(0,\tau_0)
$$
and therefore, since $\tau_0\ge x\ge \frac{y}{2}\ge 
\frac{y-x}{2}$,
$$
x^\alpha b(x)\nu(x)(y-x)\le \left(\frac{y-x}{2}\right)^\alpha
b \left(\frac{y-x}{2}\right)
\nu \left(\frac{y-x}{2}\right)(y-x)$$
$$=
2^{1-\alpha}(y-x)^\alpha
b \left(\frac{y-x}{2}\right)
 \nu \left(\frac{y-x}{2}\right) \left(\frac{y-x}{2}\right)
$$
$$
=2^{1-\alpha}(y-x)^\alpha \int_0^{(y-x)/2}
b \left(\frac{y-x}{2}\right)
\nu \left(\frac{y-x}{2}\right) ds$$
$$\le 2^{1-\alpha}(y-x)^\alpha \int_0^{(y-x)/2}b (s)\nu (s) ds\le 2^{1-\alpha}b^+(y-x)^\alpha N(y-x)\, .
$$
Now we consider the case $\xi> \tau_0$. Setting
$
\overline{\nu}^+=\displaystyle\max_{x\in[\tau_{0},T]}x^\alpha\nu (x)
$, $
\overline{\nu}_-=\displaystyle\min_{x\in ]0,T]}x^\alpha\nu (x)>0
$, 
it is
$$
x^\alpha \nu(\xi)(y-x)\le \xi^\alpha \nu(\xi)(y-x)\le \overline{\nu}^+(y-x)
=\frac{\overline{\nu}^+}{\int_0^{y-x}s^\alpha\nu(s)ds}(y-x)\int_0^{y-x}s^\alpha\nu(s)ds\, .
$$
Now there are two possibilities: $y-x\le\tau_0$, $y-x >\tau_0$. In the first possibility,
$$
\int_0^{y-x}s^\alpha\nu(s)ds
\ge \frac{1}{b^+}\int_0^{y-x}s^\alpha b(s)\nu(s)ds
\ge \frac{1}{b^+}\tau_0^\alpha b(\tau_0)\nu(\tau_0)(y-x)
$$
and therefore
$$
\frac{\overline{\nu}^+}{\int_0^{y-x}s^\alpha\nu(s)ds}(y-x)\int_0^{y-x}s^\alpha\nu(s)ds\le 
\frac{b^+\overline{\nu}^+}{\tau_0^\alpha b(\tau_0)\nu(\tau_0)}\int_0^{y-x}s^\alpha\nu(s)ds
$$
$$
\le 
\frac{b^+\overline{\nu}^+}{\tau_0^\alpha b(\tau_0)\nu(\tau_0)}\int_0^{y-x}(y-x)^\alpha\nu(s)ds=
\frac{b^+\overline{\nu}^+}{\tau_0^\alpha b(\tau_0)\nu(\tau_0)}(y-x)^\alpha N(y-x)\, ;
$$
in the second possibility,
$$
\int_0^{y-x}s^\alpha\nu(s)ds\ge \tau_0\cdot \overline{\nu}_-
$$
and therefore, using $y-x\le \frac y2\le \frac T2$,
$$
\frac{\overline{\nu}^+}{\int_0^{y-x}s^\alpha\nu(s)ds}(y-x)\int_0^{y-x}s^\alpha\nu(s)ds\le 
\frac{\overline{\nu}^+}{\tau_0\overline{\nu}_-}\cdot \frac T2\int_0^{y-x}s^\alpha\nu(s)ds$$
$$\le 
\frac{T\overline{\nu}^+}{2\tau_0\overline{\nu}_-}(y-x)^\alpha N(y-x)\, .
$$
We have therefore shown that also the second term can be estimated by the right hand side of \eqref{thesis}.

On the other hand, it is quite easy to check that also $B$ can be estimated in the same way, in fact
$$B=\left\vert\int_0^{y-x} \nu(s)[g(y)-g(y-s)]ds\right\vert\le
\int_0^{y-x} \nu(s)[g]_\alpha s^\alpha ds$$ $$\le 
[g]_\alpha\int_0^{y-x} \nu(s) (y-x)^\alpha ds=
[g]_\alpha (y-x)^\alpha N(y-x).$$
It remains to estimate $C$.
It is 
$$C=\left|\int_{y-x}^y [\nu(s-y+x)-\nu(s)][g(y)-g(y-s)]ds\right|\le [g]_\alpha
\int_{y-x}^y |\nu(s-y+x)-\nu(s)|s^\alpha ds$$

\noindent 
We now use the assumption \eqref{diffquobdd}, so that
$$
C\le [g]_\alpha
\int_{y-x}^y \left[\nu(s-y+x)-\nu(s)\right]s^\alpha ds+[g]_\alpha\int_{y-x}^y K_0(y-x)s^\alpha ds\, .
$$

Making the change of variables $\displaystyle \sigma=\frac{s}{y-x}$ in the first term, we have 

$$C \le [g]_\alpha (y-x)^{\alpha+1}
\int_{1}^{\frac{y}{y-x}}\left[ \nu((y-x)(\sigma-1)) - \nu((y-x)\sigma)\right]\sigma^\alpha d\sigma+[g]_\alpha\int_{y-x}^y K_0(y-x)s^\alpha ds$$ 

\noindent We observe that $\displaystyle x\geqslant \frac{y}{2}>\frac{x}{2}$ implies $\displaystyle \frac{y}{y-x}\geqslant 2$, therefore
\[
\begin{split}
C &\le [g]_\alpha (y-x)^{\alpha+1}\left\{
\int_{1}^{2}\left[ \nu((y-x)(\sigma-1)) - \nu((y-x)\sigma)\right]\sigma^\alpha d\sigma+\right.\\&+ \left.\int_{2}^{\frac{y}{y-x}}\left[ \nu((y-x)(\sigma-1)) - \nu((y-x)\sigma)\right]\sigma^\alpha d\sigma
\right\}+[g]_\alpha\int_{y-x}^y K_0(y-x)s^\alpha ds
\\
&:=C_{1}+C_{2}+C_{3}\end{split}
\]

We observe that the sign of $C_1$ and the sign of $C_2$ are not necessarily positive; both $C_1$ and $C_2$ will be splitted into more terms, each of them being not necessarily positive; however, all of them will be shown to be smaller than
the right hand side of \eqref{thesis}.

The first term can be estimated as follows:

\[
\begin{split}
&C_{1}=[g]_\alpha (y-x)^{\alpha+1}
\int_{1}^{2}\left[ \nu((y-x)(\sigma-1)) - \nu((y-x)\sigma)\right]
\sigma^\alpha d\sigma\\
&\leqslant 2^{\alpha}[g]_\alpha (y-x)^{\alpha+1}
\int_{1}^{2}\nu((y-x)(\sigma-1)) d\sigma
\\
&=2^{\alpha}[g]_\alpha (y-x)^{\alpha}
\int_{0}^{y-x}\nu(\tau) d\tau
\\
&= 2^{\alpha}[g]_\alpha (y-x)^{\alpha}
N(y-x)\, .
\end{split}
\]

In order to estimate $C_2$, we need to use the following inequality, which, as we are going to see, follows easily from the fact that the positive $b\nu\in AC(]0,T])$ is decreasing in $(0,\tau_0)$:
\begin{equation}\label{weakdecre}
\frac{1}{y-x}\int_x^y\nu (s)ds\le \left(\frac{b^+}{b^-}+\frac{\nu^+}{\nu^-}\right)\nu(x)\qquad \forall x,y \in (0,T), x<y
\, ,
\end{equation}
where $\nu^+=\displaystyle\max_{[\tau_{0},T]}\nu$ and $\nu^-=\displaystyle\min_{[\tau_{0},T]}\nu>0$.

In order to show \eqref{weakdecre}, let us fix $x,y \in (0,T), x<y$. If $y\le\tau_0$, using that $b\nu$ is decreasing in $(0,\tau_0)$, we have
$$
\frac{1}{y-x}\int_x^y\nu (s)ds\le
\frac{1}{b^-}\frac{1}{y-x}\int_x^y b(s)\nu (s)ds
\le \frac{1}{b^-} b(x)\nu (x)\le  \left(\frac{b^+}{b^-}+\frac{\nu^+}{\nu^-}\right)\nu(x).
$$
If $\tau_0\le x$, 
$$
\frac{1}{y-x}\int_x^y\nu (s)ds\le \nu^+ = \frac{\nu^+}{\nu^-}\nu^- \le 
\frac{\nu^+}{\nu^-}\nu (x)\le \left(\frac{b^+}{b^-}+\frac{\nu^+}{\nu^-}\right)\nu(x).
$$
The last case is $x<\tau_0<y$, where the following inequalities hold:
$$
\frac{1}{y-x}\int_x^y\nu (s)ds=\frac{1}{y-x}\int_x^{\tau_0}\nu (s)ds+\frac{1}{y-x}\int_{\tau_0}^y\nu (s)ds$$
$$\le\frac{1}{\tau_0-x}\int_x^{\tau_0}\nu (s)ds+\frac{1}{y-\tau_0}\int_{\tau_0}^y\nu (s)ds\le \frac{b^+}{b^-}\nu (x)+\frac{\nu^+}{\nu^-}\nu (x)=\left(\frac{b^+}{b^-}+\frac{\nu^+}{\nu^-}\right)\nu(x)\, .
$$

Using the change of variable $\sigma=\tau+1$, the integration by parts formula, $\nu >0$, \eqref{weakdecre}
and Lemma \ref{secondolemma}, we now estimate $C_2$ as follows:

\begin{equation*}
\begin{split}
&C_{2}=[g]_\alpha (y-x)^{\alpha+1}
\int_{2}^{\frac{y}{y-x}}\left[ \nu((y-x)(\sigma-1)) - \nu((y-x)\sigma)\right]\sigma^\alpha d\sigma\\
&=[g]_\alpha (y-x)^{\alpha+1}
\int_{1}^{\frac{x}{y-x}}\left[ \nu((y-x)\tau) - \nu((y-x)(\tau+1))\right](\tau+1)^\alpha d\tau\\
&=[g]_\alpha (y-x)^{\alpha+1}
\int_{1}^{\frac{x}{y-x}}(\tau+1)^\alpha d\left(\frac{1}{y-x}\int_{(y-x)(\tau+1)}^{(y-x)\tau}\nu (s)ds\right)\\
&=[g]_\alpha (y-x)^{\alpha+1}
\left\{
\left.
\frac{(\tau+1)^\alpha}{y-x}\int_{(y-x)(\tau+1)}^{(y-x)\tau}\nu (s)ds
\right\vert^{\tau=\frac{x}{y-x}}_{\tau=1}\right.\\
&\qquad\qquad\qquad\qquad\qquad\qquad\qquad-\left.
\frac{1}{y-x}
\int_{1}^{\frac{x}{y-x}}\left(\int_{(y-x)(\tau+1)}^{(y-x)\tau}\nu (s)ds\right)\alpha (\tau+1)^{\alpha-1}d\tau
\right\}\\
&=[g]_\alpha (y-x)^{\alpha}
\left\{
\left.
(\tau+1)^\alpha\int_{(y-x)(\tau+1)}^{(y-x)\tau}\nu (s)ds
\right\vert^{\tau=\frac{x}{y-x}}_{\tau=1}\right.\\
&\qquad\qquad\qquad\qquad\qquad\qquad\qquad+\left.
\int_{1}^{\frac{x}{y-x}}\left(\int_{(y-x)\tau}^{(y-x)(\tau+1)}\nu (s)ds\right)\alpha (\tau+1)^{\alpha-1}d\tau
\right\}\\
&=[g]_\alpha (y-x)^{\alpha}
\left\{
\left(\frac{y}{y-x}\right)^{\alpha}\int_{y}^{x}\nu(s)ds-2^{\alpha}\int_{2(y-x)}^{y-x}\nu(s)ds\right.\\
&\qquad\qquad\qquad\qquad\qquad\qquad\qquad+\left.
\int_{1}^{\frac{x}{y-x}}\left[\int_{(y-x)\tau}^{(y-x)(\tau+1)}\nu (s)ds\right]\alpha (\tau+1)^{\alpha-1}d\tau
\right\}\\
&\leqslant [g]_\alpha (y-x)^{\alpha}\left\{2^{\alpha}\int_{y-x}^{2(y-x)}\nu(s)ds
+c\left(b(\cdot),\nu(\cdot)\right)
\int_{1}^{\frac{x}{y-x}}\left[
(y-x)\nu ((y-x)\tau)
\right]\alpha (\tau+1)^{\alpha-1}d\tau
\right\}\\
&\leqslant [g]_\alpha (y-x)^{\alpha}\left\{2^{\alpha}c\left(b(\cdot),\nu(\cdot),T\right)N(y-x)\right.\\
&\qquad\qquad\qquad+\left.
 c\left(b(\cdot),\nu(\cdot)\right)(y-x)\alpha\frac{N(y-x)}{y-x}\int_{1}^{\frac{y}{y-x}}(\tau+1)^{\alpha-1}\frac{\nu((y-x)\tau)\,(y-x)}{N(y-x)}d\tau\right\}
\end{split}
\end{equation*}

\noindent
By Corollary \ref{usatodavvero} the integral inside the parenthesis is bounded by a constant independent of $y-x$, depending only on $b(\cdot),\nu(\cdot),T,\alpha+\beta$. Hence the estimate becomes
$$C_{2}\le [g]_\alpha (y-x)^{\alpha}\left\{2^{\alpha}c\left(b(\cdot),\nu(\cdot),T\right)N(y-x) + \alpha \textit{c}\left(b(\cdot),\nu(\cdot),T,\alpha+\beta\right) N(y-x) \right\}$$
and therefore also $C_2$ is estimated by the right hand side of \eqref{thesis}.

Finally, we need to estimate $C_3$.

We have 
$$C_3=[g]_\alpha\int_{y-x}^y K_0(y-x)s^\alpha ds\le [g]_\alpha K_0 \int_0^T s^\alpha ds (y-x)$$
$$=\textit{c}\left(\nu(\cdot),T,\alpha\right)[g]_\alpha(y-x)=\textit{c}\left(\nu(\cdot),T,\alpha\right)[g]_\alpha
\frac{(y-x)^{1-\alpha}}{N(y-x)}(y-x)^{\alpha}N(y-x)$$
$$\le 
\textit{c}\left(b(\cdot),\nu(\cdot),T,\alpha\right)[g]_\alpha(y-x)^{\alpha}N(y-x)
$$
where the last inequality follows from the fact that from \eqref{bastaquestaipotesi} and 
from $\beta<1-\alpha$ it follows that 
$$
\frac{1}{x^{1-\alpha}}\int_0^x b(s)\nu(s)ds\searrow\,\, \textrm{in}\,\, (0, \tau_{0})\, ,
$$
hence $N(x)/x^{1-\alpha}$ is bounded below in $(0,T)$ by a positive constant, i.e. its reciprocal is bounded above by a positive constant (depending only on $b$, $\nu$ and $\alpha$).
\end{proof}

\begin{rem}
In the case $b\equiv 1$, $\nu(\sigma)=\sigma^{\beta-1}$, $\tau_{0}= T$, Theorem \ref{casoholder} gives back Theorem 4.2.1 in \cite{gv}. Another interesting case is 
\begin{equation}\label{exeimportante}
\nu(\sigma)=\frac{1}{\sigma(\log \left(1/\sigma\right))^2}\, ,\, \sigma \,\, \textrm{small}
\end{equation}
which satisfies the assumption \eqref{bastaquestaipotesi} of Theorem \ref{casoholder} for \it any $0<\alpha<1$, \rm for \it any $0<\beta<1-\alpha$, \rm with $b\equiv 1$. \rm Of course those positive functions $\nu\in AC(]0,T])\cap L^1(0,T)$, which are just \it equivalent \rm to the right hand side of 
\eqref{exeimportante} only in a neighborhood of the origin, and then not blowing up ``too much'' (as, for instance, the function $\I(t)$ in 
\eqref{eq:i}, whose derivative -- see \eqref{derivat} -- is again a Volterra function which is bounded above), are examples for Theorem \ref{casoholder} and in such cases the resulting regularity for $J_\nu$ is the same as that one given for \eqref{exeimportante}.
\end{rem}

\begin{rem}
From the proof of Theorem \ref{casoholder} it is clear that the assumption \eqref{diffquobdd} can be weakened as follows:
$$
|\nu (x) - \nu(y)| \le  \nu (x) - \nu(y) + K_0 (y-x)^{\alpha}N(y-x)\,\,\textrm{for some}\,\,K_0>0,\, \forall x,y\in (0,T), \frac y2\le x<y.$$
\end{rem}

\begin{rem} For a given $\nu$ satisfying the assumptions of Theorem \ref{casoholder} one may look for the \it best \rm regularity action for $J_\nu$, i.e. one may look for the greatest $\alpha$ satisfying the assumptions of the theorem or, equivalently, for the smallest $\beta$ satisfying \eqref{bastaquestaipotesi} (in the case of 
the classical spaces of H\"older continuous functions,
the inclusions between the spaces are easy and well known; for a recent book on this topic see \cite{refioren}): this problem is linked to the notion of Boyd indices (see e.g. \cite{Maligr}). For a short survey including a bibliography on this topic, and for a ``concrete'' way to compute them for explicit examples, see e.g. \cite{fiokrb1},\cite{fiokrb2}.
\end{rem}

\begin{rem}
It is interesting to note that in the paper \cite{SM} (see also \cite{cs}) the authors prove a result of the same type as Theorem 3.1, where kernels more general than powers are considered. However, in \cite{SM} the assumption to belong to a certain class $V_\lambda$, $0<\lambda<1$, implies that the kernel enjoys a higher integrability property (in fact, if $k\in V_\lambda$, it is $k(x)\le c\, x^{-\lambda}$ around zero), hence kernels like our model example $\I$, discussed in Section \ref{sec-volt}, cannot be considered.
\end{rem}


\section{Regularization in $L^{p}$ spaces}\label{sectionellepi}

\noindent
We begin some background on Young's functions and Orlicz spaces. In the following a convex function $A$ defined on $[0,\infty[$ is said to be a Young's function if it is convex and such that $A(0)=0$, $A(x)>0$ for $x>0$. This assumption implies that Young's functions are strictly increasing and invertible, so that it makes sense to consider its inverse $A^{-1}$, defined in $[0,\infty[$. The Orlicz space $L^A(0,T)$ (here $T$ is a fixed positive real number) is the Banach function space of all real-valued (Lebesgue) measurable functions $f$ on $(0,T)$ such that 
$$
\|f\|_{L^A(0,T)}=\|f\|_{A}:=\inf \left\{ \lambda>0 \, : \, \int_0^T A\left(\frac{|f(t)|}{\lambda}\right)dt\le 1 \right\}<\infty 
$$
(here we use the convention $\inf\emptyset=+\infty$). In the special case $A(t)=t^p$, $1\le p<\infty$, the Orlicz space reduces to the familiar Lebesgue space. For essentials about Orlicz spaces and Banach function spaces the reader may refer to \cite{CF}, Sections 2.10.2 and 2.10.3 (and references therein for extensive treatments). A well known result of the theory is that if two Young's functions $A$, $B$ are such that
$$
A(c_1x)\le B(x) \le A(c_2x) \qquad \textrm{for} \,\, x\,\, \textrm{large}\, ,
$$
then, in spite $\|\cdot\|_A$, $\|\cdot\|_B$ may be different, the spaces themselves (namely, the set of the functions such that the norms are finite) coincide. In particular, the spaces are completely determined by the values of the Young's functions assumed for $x$ large.

For measurable functions $f\not\equiv 0$ on $(0,T)$, the decreasing rearrangement $f^*$ is defined by the right continuous inverse of $t\to \mu(t)=meas(s\in (0,T) \, : |f(s)|>t)$, i.e. $f^*(t)=\inf\{ \lambda >0 \, : \mu(\lambda)>t \}$. Orlicz spaces $L^A(0,T)$ are rearrangement-invariant: this means, in particular, that the norm is not affected after the action of the decreasing rearrangement operator: $\|f\|_A=\|f^*\|_A$.

We may state the following

\begin{teo}\label{casoLp}
Let $1<p<\infty$, $0<T<\infty$, and let $A$ be a Young's function. 
If $\nu\in L^1(0,T)$, $\nu >0$, is such that
\begin{equation}\label{ipotesidecrescente}
x\to \nu (x) \,\,\textrm{is decreasing in} \,\,(0,\tau_0) \,\,\textrm{for some}\,\, 0<\tau_{0}\leq T\, ,
\end{equation}
\begin{equation}\label{ipotesilimitato}
x\to \nu (x) \,\,\textrm{is bounded in} \,\,(\tau_0, T) \quad \textrm{(in the case $\tau_0<T$)}\, ,
\end{equation}
\begin{equation}\label{marcin}
N(t)\le c_{\nu} t A^{-1}\left(\frac1t\right)
\,\,\,\,\textrm{for some}\,\,c_\nu>0,\quad \forall t \in (0,\tau_0)
\, ,
\end{equation}
where
$$
N(t)=\int_{0}^{t}\nu(\sigma)d\sigma\, ,
$$
then, setting
$$
J_\nu g(t)=\int_0^t \nu(t-s)g(s)ds\qquad t\in (0,T),
$$
it is 
\begin{equation}\label{thesis4Lp}
\| J_\nu g\|_C\le c\left(\nu(\cdot),A(\cdot), p,T\right) \|g\|_p\, ,
\end{equation}
where $C$ is the Young's function (Proof of Lemma 4.2 in \cite{oneilok}) defined by 
\begin{equation}\label{eccoC}
C^{-1}(x)=\int_0^x t^{-2+\frac{1}{p}}A^{-1}(t)dt,\quad x\geq0.
\end{equation}
\end{teo}

Before giving the proof of the theorem, which is a quite easy consequence (in fact, an application) of a classical result about fractional integration in Orlicz spaces, we highlight a couple of examples which are relevant for this paper.

\begin{example} Let $1<p<\infty$ and $0<\alpha<1/p$, and let $\nu(s)=s^{\alpha-1}\in L^{r}(0,T)\subset L^1(0,T)$, for all $1<r<1/(1-\alpha)$. Then $A(x)=x^{1/(1-\alpha)}$ satisfies \eqref{marcin}, hence $C$ given by \eqref{eccoC} is $C(x)=x^{p/(1-\alpha p)}$, and $J_{\nu}g\in L^{p/(1-\alpha p)}$ for all $g\in L^p$. This special case gives back the refined version of the continuity property for the Abel operator in $L^p$ spaces, see
Theorem 4.1.3 in \cite{gv}.
\par
Notice that when $\alpha$ approaches $0$, the exponent $r$ of integrability of $\nu$ approaches $1$, and the exponent $p/(1-\alpha p)$ of integrability of $J_{\nu}g$ approaches $p$, which means no gain of integrability: in the framework of the Lebesgue spaces, kernels in $L^1$ which do not possess the higher integrability property (see e.g. next two examples and, in particular, the Volterra function $\I(t)$) are not able to improve the integrability through the operator $J_{\nu}$.
\end{example}

\begin{example} Let $1<p<\infty$ and let $\beta>1$, and let $\nu(s)=\frac{1}{s\log^\beta \left(\frac1s\right)}\in L^1(0,1/2)$. Then 
$$
N(t)=\int_{0}^{t}  \frac{1}{\sigma \log^\beta \left(\frac1\sigma\right)}d\sigma = \frac{1}{\beta-1}\log^{1-\beta} \left(\frac1t\right)\, ,
$$
hence
$$
A(x)\approx x\log^{\beta-1}x \qquad \textrm{for} \,\, x\,\, \textrm{large}
$$
satisfies \eqref{marcin}. The Young's function $C$ given by \eqref{eccoC} is 
$$
C(x)\approx x^p\log^{p(\beta-1)}x \qquad \textrm{for} \,\, x\,\, \textrm{large}\, ,
$$
and $J_{\nu}g\in L^{C}$ for all $g\in L^p$. 
It is interesting to note that the kernel does not belong to any Lebesgue space $L^p$ with $p>1$, and that the logarithm in the expression of $C$ (which has a positive power and therefore it is divergent at infinity) represents an Orlicz gain of integrability for $J_{\nu}g$.
\end{example}

For any kernel considered in Theorem \ref{casoLp}, the existence of a Young function $A$ satisfying \eqref{marcin} can be easily established; moreover, for any $\nu$ 
the function $J_{\nu}g$ always enjoys an Orlicz gain of integrability with respect to $g$: this is the heart of the following simple result, which is consequence of standard statements of Orlicz spaces theory, namely, of the fact that 
any function $L^1(0,T)$ is always in some Orlicz space 
$L^\Psi(0,T)$ strictly contained in $L^1(0,T)$ (see e.g. \cite{krasnorut}, p.60), of the H\"older's inequality in Orlicz spaces (see e.g. \cite{adams}, 8.11 p. 234)
\begin{equation}\label{holdOrlicz}
\int_0^Tfg\, ds\le 2\|f\|_\Psi\|g\|_{\tilde \Psi}
\end{equation}
where $\tilde \Psi$ is the Young function defined by $\tilde \Psi(s)=\max_{t\ge 0}(st-\Psi(t))$, of the equivalences (see e.g. \cite{adams}, (7) p. 230 and \cite{Shar}, respectively)
\begin{equation}\label{equiv4tildeA}
(\tilde \Psi)^{-1}(t)\approx \frac{t}{\Psi^{-1}(t)}\, \quad t>0
\end{equation}
\begin{equation}\label{fundfunct}
\|\one_{(0,t)}\|_\Psi= \frac{1}{\Psi^{-1}(1/t)}\, \quad 0<t\le T\, ,
\end{equation}
and finally of 
\begin{equation}\label{inclusionOrlicz}
L^C(0,T) \,\,\subset\,\, L^\Psi(0,T) \,\,\Leftrightarrow\,\, \Psi(t)\le C(kt) \,\,\textrm{for some}\,\, k>0, \,\,\textrm{for}\,\, t \,\,\textrm{large}\, .
\end{equation}

\begin{pro}
In the assumptions of Theorem \ref{casoLp}, for every $\nu$ there exists a Young function $A$ satisfying \eqref{marcin}, and therefore $L^C(0,T)$ is strictly contained in $L^p(0,T)$.
\end{pro}

\begin{proof}
Let $A$ be a Young function such that $\nu\in L^A(0,T)$, $A(x)/x$ being increasing and divergent at infinity.
By \eqref{holdOrlicz}, \eqref{fundfunct}, \eqref{equiv4tildeA} respectively, 
$$
N(t)=\int_{0}^{t}\nu(\sigma)d\sigma=
\int_{0}^{T}\nu(\sigma)\chi_{(0,t)}d\sigma
\le 2 \|\nu\|_A\|\one_{(0,t)}\|_{\tilde A}
= c_\nu \cdot \frac{1}{(\tilde A)^{-1}(1/t)}
\approx c_\nu t A^{-1}\left(\frac1t\right)
$$
As to the second part of the statement, from \eqref{eccoC} we get that setting, for large $s$, $x=C(s)$ and $t=A^{-1}(x)$ (note that $x$ and $t$ are large as well)
$$
\frac{C(s)}{s^p}=\frac{x}{C^{-1}(x)^p}=\left(
\frac{x^{1/p}}{C^{-1}(x)}
\right)^p=
\left(
\frac{x}{A^{-1}(x)}
\right)^p=\left(
\frac{A(t)}{t}
\right)^p\nearrow \infty\, ,
$$
and using both implications in \eqref{inclusionOrlicz}, we get the assertion.
\end{proof}

\begin{example}
It is immediate to realize that the statement of Theorem \ref{casoLp} remains true if $\nu$ is replaced by any function equivalent to $\nu$ in a neighborhood of the origin. Hence all the previous remark still holds if $\nu$ is replaced by the function $\I(t)$ in 
\eqref{eq:i}.
\end{example}

The proof of Theorem \ref{casoLp} will follow as consequence of the following result appeared in Sharpley (\cite[Theorem 3.8]{Shar}), in the more abstract setting of general convolution operators, defined in \cite{oneilok}. In this latter paper our operator $J_\nu$, which goes back to \cite{hl}, is explicitly mentioned as example (see the end of the Section IV therein).
 Here we state it in a more convenient form, and using our notation:

\begin{teo}\label{bysharpley}
Suppose that $A$, $B$ are Young's functions such that 
\begin{equation}\label{hypot1}
xB'(x)\le c_B B(x) 
\qquad \textrm{for some}\,\, c_B>1\, , \,\, \textrm{for} \,\, x\,\, \textrm{large}
\end{equation}
and 
\begin{equation}\label{hypot2}
c_C C(x)\le xC'(x) 
\qquad \textrm{for some}\,\, c_C>1\, , \,\, \textrm{for} \,\, x\,\, \textrm{large}\, ,
\end{equation}
where $C$ is the Young's function defined by 
\begin{equation}\label{eccoCshar}
\frac{1}{xC^{-1}(x)}=\frac{1}{A^{-1}(x)}\cdot \frac{1}{B^{-1}(x)}\qquad \textrm{for} \,\, x\,\, \textrm{large}\, .
\end{equation}
Then,
\begin{equation}\label{thesis4Lpshar}
\| J_\nu g\|_C\le c_{A,B}\sup_x \left\{ \frac{\nu^{**}(x)}{A^{-1}\left(\frac1x\right)}\right\}
 \|g\|_B\, ,
\end{equation}
where ($J_\nu$ is defined by \eqref{J_nu} and) $\nu^{**}$ denotes the averaged rearrangement of $\nu$, defined by 
$$
\nu^{**}(t)=\frac1t\int_0^t\nu^*(s)ds\, .
$$
\end{teo}

\it Proof of Theorem \ref{casoLp}. \rm 

Setting $B(t)=t^p$, \eqref{hypot1} is obviously satisfied with equality and $c_B=1$; from \eqref{eccoCshar} we get that if $C$ is defined by
$$ 
\frac{1}{xC^{-1}(x)}=\frac{1}{A^{-1}(x)}\cdot \frac{1}{x^{1/p}}\qquad \textrm{for} \,\, x\,\, \textrm{large}\, ,
$$
i.e. if \eqref{eccoC} holds, then \eqref{hypot2} is satisfied: in fact, from the convexity of the Young function $A$, the function $A^{-1}(x)/x$ is decreasing, hence from \eqref{eccoC} we deduce that also the following ones are decreasing:
$$
\frac{C^{-1}(x)}{x^{1/p}}\,\, , \,\, \frac{x}{C(x)^{1/p}}\,\, , \,\,  \frac{x^p}{C(x)}\,\, , \,\, 
$$
and therefore $C(x)/x^p$ is increasing, from which \eqref{hypot2} is satisfied with $c_C=p$. We are therefore allowed to apply
Theorem \ref{bysharpley}.
 
From \eqref{ipotesidecrescente} and \eqref{ipotesilimitato} it follows that for small values of $t$ it is $\nu(t)=\nu^*(t)$, hence, by \eqref{marcin}, for $x$ small it is
$$
 \frac{\nu^{**}(x)}{A^{-1}\left(\frac1x\right)}=\frac{N(x)}{xA^{-1}\left(\frac1x\right)}\le c_{\nu}\, ,
$$
and of course the same conclusion holds (for a possibly different $c_\nu$) for all $x\in (0,T)$. Hence the supremum in the right hand side of \eqref{thesis4Lpshar} is a finite constant $c\left(\nu(\cdot), A(\cdot), p,T\right)$, from which \eqref{thesis4Lp} follows.

\medskip
Finally, the case $p=\infty$ is considered in the next result, where the regularizing effect of $J_{\nu}$ is expressed through continuity.

\begin{pro}
 \label{pro:infinity}
 If $\nu>0$ and $\nu\in L^1(0,T)$, then for any $g\in L^\infty(0,T)$ there results
 \[
  J_\nu g\in C^0([0,T])
 \]
 and
 \begin{equation}
 \label{eq:est_I_inf}
 \|J_\nu g\|_{L^\infty(0,T)} \leq N(T) \|g\|_{L^\infty(0,T)},
\end{equation}
where $N$ is the integral function of $\nu$ defined by \eqref{eq:int_funct_nu}.
\end{pro}

\begin{proof}
 Recalling \eqref{eq:op_I} and \eqref{eq:Nasympt}, \eqref{eq:est_I_inf} is immediate. Then, it is left to prove that $J_\nu g$ is continuous. To this aim, fix $x_0\in[0,T[$ and $x\in\;]x_0,T]$. Easy computations yield
 \begin{align*}
  J_\nu g(x)-J_\nu g(x_0) = & \, \int_0^T\nu(x-s)\one_{[x_0,x]}(s)g(s)ds+\\[.2cm]
                            & \, -\int_0^T\big(\nu(x_0-s)-\nu(x-s)\big)\one_{[0,x_0]}(s)g(s)ds
 \end{align*}
 and hence
 \begin{align*}
 \left |J_\nu g(x)-J_\nu g(x_0)\right|\leq & \,\int_0^T\nu(x-s)\one_{[x_0,x]}(s)|g(s)|ds+ \\
                      & \,+\int_0^T|\nu(x_0-s)-\nu(x-s)|\one_{[0,x_0]}(s)|g(s)|ds \\
                      \leq & \,N(x-x_{0})\|g\|_{L^\infty(0,T)}+\|g\|_{L^\infty(0,T)} \int_\mathbb{R}|\nu_e(x_0-s)-\nu_e(x-s)|ds, 
 \end{align*}
 where
 \[
  \nu_e(s)=\left\{
  \begin{array}{ll}
   \nu(s), & \text{if }s\in\;]0,T[\,,\\[.2cm]
   0,      & \text{if }s\in\R\backslash\;]0,T[\,.
  \end{array}
  \right.
 \]
 Therefore, the first term converges  to zero by the continuity of $N$ and the second term converges to zero by the \emph{mean continuity property} (see \cite{okiko}). Since the same holds if $x<x_0$, one has $J_\nu g(x)\to J_\nu g(x_0)$, which concludes the proof.

\end{proof}


\section{Contraction in Sobolev spaces}
\label{sob!}

\noindent
We start with some basics on Sobolev spaces with \emph{fractional} index. Let $-\infty \leq a < b \leq +\infty$ and $\theta \in (0,1)$. We denote by $H^\theta(a,b)$ the Sobolev space defined by
\[
 H^\theta(a,b) := \{g\in L^2(a,b):[g]_{\dot{H}^\theta(a,b)}^2<\infty\},
\]
where
\[
 [g]_{\dot{H}^\theta(a,b)}^2 := \int_{[a,b]^2}\frac{|g(x)-g(y)|^2}{|x-y|^{1+2\theta}}dydx.
\]
This is a Hilbert space with the natural norm
\[
 \|g\|_{H^\theta(a,b)}^2 := \|g\|_{L^2(a,b)}^2 + [g]_{\dot{H}^\theta(a,b)}^2.
\]
When $a = -\infty$ and $b = +\infty$, $H^\theta(\R)$ can be equivalently defined using the Fourier transform (see \cite{DiN}); that is, if we define the Fourier transform as
\[
 \widehat{g}(k):=\frac{1}{\sqrt{2\pi}}\int_\R e^{-ikx}g(x)dx,
\]
then there exist two constants $c_{1,\theta},\,c_{2,\theta} > 0$ such that 
\begin{equation}
\label{eq:normeq}
 c_{1,\theta}\|(1+k^2)^{\theta/2}\,\widehat{g}\,\|_{L^2(\R)} \leq \|g\|_{H^\theta(\R)} \leq c_{2,\theta}\|(1+k^2)^{\theta/2}\,\widehat{g}\,\|_{L^2(\R)}.
\end{equation}

\begin{rem}
 We recall that with $\theta=1$ we mean the Sobolev space $H^1$, with the usual definition, whereas, with a little abuse, $\theta=0$ is an equivalent notation for $L^2$.
\end{rem}

Now, before stating the main theorem of this section, we recall a result on \emph{truncation/extension} of functions in $H^\theta(0,T)$.

\begin{lem}
\label{lem-ext}
 Let $g\in H^\theta(0,T)$, $0\leq\theta\leq1$, and set
 \begin{equation}
  \label{eq:ext}
  g_e(x) = \left\{
  \begin{array}{ll}
   g(x),    & \text{if } x\in[0,T],\\[.3cm]
   g(2T-x), & \text{if } x\in\;]T,2T],\\[.3cm]
   0,       & \text{if } x\in\R\backslash[0,2T].
  \end{array}
  \right.
 \end{equation}
The following holds:
 \begin{itemize}
  \item[(i)] if $\theta\in[0,1/2[\,$, then $g_e\in H^\theta(\R)$;\\[-.3cm]
  \item[(ii)] if $\theta\in\;]1/2,1]$ and $g(0)=0$, then $g_e\in H^\theta(\R)$.
 \end{itemize}
Moreover, in both cases, there exists a constant $c_\theta>0$ (independent of $g$ and $T$) such that
 \begin{equation}
 \label{eq:stimaf}
  \|g_e\|_{H^\theta(\R)} \leq c_\theta \|g\|_{H^\theta(0,T)}.
 \end{equation}
\end{lem}

\begin{proof}
 The proof is a straightforward application of Lemma 2.1 in \cite{CFNT}. Cases $\theta=0,1$ are trivial. Consider, then, an arbitrary $\theta\in(0,1)$. First, we can easily check that $\|g_e\|_{L^2(0,2T)}^2 = 2\|g\|_{L^2(0,T)}^2$ and that, with some change of variables,
 \[
  [g_e]_{\dot{H}^\theta(0,2T)}^2 = 2[g]_{\dot{H}^\theta(0,T)}^2 + 2\int_{[0,T]^2}\frac{|g(x)-g(y)|^2}{|x+y-2T|^{1+2\theta}}dydx.
 \]
 Moreover, since $|x+y-2T| \geq |x-y|$ for every $(y,x) \in [0,T]^2$,
 \[
  \int_{[0,T]^2}\frac{|g(x)-g(y)|^2}{|x+y-2T|^{1+2\theta}}dydx \leq [g]_{\dot{H}^\theta(0,T)}^2.
 \]
 Hence $[g_e]_{\dot{H}^\theta(0,2T)}^2 \leq 4[g]_{\dot{H}^\theta(0,T)}^2$, so that
 \begin{equation}
 \label{eq:extradd}
  \|g_e\|_{H^\theta(0,2T)} \leq 2\|g\|_{H^\theta(0,T)}.
 \end{equation}
 Now, from Lemma 2.1 of \cite{CFNT}, we know that if $\theta\in\;]0,1/2[\,$, then there exists $c_\theta > 0$ such that
 \begin{equation}
 \label{eq:exterre}
  \|g_e\|_{H^\theta(\R)} \leq c_\theta \|g_e\|_{H^\theta(0,2T)}.
 \end{equation}
 On the other hand, the same lemma shows that \eqref{eq:exterre} holds even if $\theta\in\;]1/2,1[\,$, provided that $g(0)=0$ (since this entails by definition $g_e(0)=g_e(2T)=0$). Combining \eqref{eq:exterre} and \eqref{eq:extradd}, the proof is complete.
\end{proof}

\begin{rem}
 Note that, when $\theta>1/2$, the assumption $g(0)=0$ is meaningful as $g$ is continuous by Sobolev embeddings (\cite{DD,DiN}). What is more, this requirement is mandatory, since otherwise $g_e$ might not preserve continuity on $\R$.
\end{rem}

\begin{rem}
 We also stress that the case $\theta=1/2$ is not managed by Lemma \ref{lem-ext} since Lemma 2.1 in \cite{CFNT} is not valid in general for this choice of $\theta$, due to the failure of Hardy inequality (see e.g. \cite{kp}).
\end{rem}

Then, we can state the main result of this section. We recall that, as in the previous sections, $J_\nu$ denotes the operator defined by \eqref{J_nu} and $N$ the integral function of the kernel defined by \eqref{eq:int_funct_nu}.

\begin{teo}
\label{contr_lemma}
 If $\theta\in[0,1]$ and $\nu>0$, $\nu\in L^{1}(0,T)$, then there exists a constant $c_\theta>0$ such that
\begin{equation}
\label{eq:est_I}
\|J_\nu g\|_{H^{\theta}(0,T)} \leq c_\theta N(T) \|g\|_{H^{\theta}(0,T)},\quad\forall g\in H^\theta(0,T),\quad\theta<1/2.
\end{equation}
Moreover, \eqref{eq:est_I} is valid also when $\theta>1/2$, provided that $g$ satisfies $g(0)=0$.
\end{teo}

\begin{proof}
 Fix $\theta\in[0,1]\backslash\{1/2\}$. Then, let again
 \[
  \nu_e(x):=\left\{
  \begin{array}{ll}
   \nu(x), & \text{if }x\in\;]0,T[\,,\\[.3cm]
   0,      & \text{if }x\in\R\backslash\;]0,T[
  \end{array}
  \right.
 \]
 and, for any $g\in H^\theta(0,T)$, define
 \[
  f(x) := \int_0^x\nu_e(x-s)g_e(s)ds,
 \]
 where $g_e$ is the extension of $g$ obtained via Lemma \ref{lem-ext}. Note that \eqref{eq:stimaf} applies if either $\theta<1/2$ or $\theta>1/2$ with, in this second case, the further assumption that $g(0)=0$. As $f(x)=(J_\nu g)(x)$ for all $x\in[0,T]$,
 \begin{equation}
 \label{eq:da_T_a_R}
 \|J_{\nu}g\|_{H^\theta(0,T)} = \|f\|_{H^\theta(0,T)} \leq \|f\|_{H^\theta(\R)}.
 \end{equation}
 Now, by definition
 \[
  \widehat{f}(k) = \frac{1}{\sqrt{2\pi}}\int_\R e^{-ikx} \int_\R \nu_e(x-s)g_e(s)\one_{[0,x]}(s)ds\,dx.
 \]
 Since $\one_{[0,x]}(s) = H(s) - H(s-x)$, where $H$ denotes the Heaviside function,
 \[
  \widehat{f} = \widehat{\nu_e*(g_eH)} - \widehat{(\nu_eH_r)*g_e},
 \]
 with $H_r(s) = H(-s)$. Consequently, by well known properties of the Fourier transform,
 \[
  \widehat{f} = c\left(\widehat{\nu_e}\widehat{g_eH} - \widehat{\nu_eH_r}\widehat{g_e}\right).
 \]
 Thus, noting that $g_e(x)H(x)=g_e(x)$ and $\nu_e(x)H_r(x)=0$,
 \begin{equation}
 \label{eq:gtras}
  \widehat{f} = c\,\widehat{\nu_e}\widehat{g_e}.
 \end{equation}
 Then, combining \eqref{eq:da_T_a_R}, \eqref{eq:normeq} and \eqref{eq:gtras},
 \begin{equation}
 \label{eq:g_norm}
 \|J_{\nu}g\|_{H^\theta(0,T)}^2 \leq c\int_\R(1+k^2)^\theta|\widehat{\nu_e}(k)|^2\,|\widehat{g_e}(k)|^2dk.
 \end{equation}
 Moreover, we observe that $|\widehat{\nu_e}(k)| \leq c N(T)$ and, plugging into \eqref{eq:g_norm}, that
 \[
  \|J_{\nu}g\|_{H^\theta(0,T)}^2 \leq c\cdot N^2(T)\|g_e\|_{H^\theta(\R)}^2.
 \]
 Combining with \eqref{eq:stimaf}, \eqref{eq:est_I} follows.
 \end{proof}
 
 \begin{rem}
  We observe that the ``contractive'' effect of $J_\nu$, pointed out in the Introduction, is in the fact that $N(T)\to0$, as $T\to0$. It entails that on small intervals the operator ``shrinks'' the norm of the argument function by a factor that gets smaller whenever $T$ gets smaller.
 \end{rem}

 \begin{rem}
 Note that in the previous theorem, when $\theta>1/2$ the assumption $g(0)=0$ cannot be removed, since otherwise the result is false. If one dropped this requirement, indeed, then the statement would imply that $N(x)=(J_\nu 1)(x)$ belongs to $H^\theta(0,T)$, which cannot hold in general. A remarkable counterexample is given by the case $\nu=\I$, where one can prove that $N=\NN$ does not belong to $H^\theta(0,T)$ for any $\theta>1/2$ (see Lemma \ref{lem:N}).
\end{rem}
 
The case $\theta=1/2$ is far more awkward since no ``extension-to-zero'' result, such as Lemma \ref{lem-ext}, is available. However, in the case $\nu=\I$, we can state an analogous for Theorem \ref{contr_lemma}. In order to prove it, it is though required a further investigation of the behavior of the integral function $\I$.
 
\begin{lem}
 \label{lem:N}
 The function $\NN$ defined by \eqref{enne} does not belong to $H^\theta(0,T)$ for any $\theta\in\;]1/2,1]$. On the other contrary, it belongs to $ H^\theta(0,T)$ for every $\theta\in[0,1/2]$.
\end{lem}

\begin{proof}
 The first part is immediate. In fact, if $\NN\in H^\theta(0,T)$, then it should be H\"older continuous in $[0,T]$ as well (see \cite{DD}). However, one can easily see that this is not the case by \eqref{eq:Nasympt}.
 
 On the other hand, one easily checks that $\NN\in L^2(0,T)$, so that it is left to prove that $[\NN]_{\dot{H}^{1/2}(0,T)}<\infty$ (since $H^\theta(0,T)\subset H^{1/2}(0,T)$ for all $\theta\in[0,1/2[\,$, by \cite{DiN}). An easy computation shows that
 \[
  [\NN]_{\dot{H}^{1/2}(0,T)}^2  = 2\int_0^T\int_0^{x/2}\left|\frac{\NN(x)-\NN(y)}{x-y}\right|^2dy\,dx+2\int_0^T\int_{x/2}^x\left|\frac{\NN(x)-\NN(y)}{x-y}\right|^2dy\,dx.                              
 \]
 Looking at the first integral and recalling that $\NN$ is increasing, we find
 \[
  \left|\frac{\NN(x)-\NN(y)}{x-y}\right|^2\leq4\,\frac{\NN^2(x)}{x^2},\quad\forall y\in(0,x/2).
 \]
 Hence, combining with \eqref{eq:Nasympt} and \eqref{eq:Iasympt},
 \[
  2\int_0^T\int_0^{x/2}\left|\frac{\NN(x)-\NN(y)}{x-y}\right|^2dy\,dx\leq c\int_0^T\frac{\NN^2(x)}{x}dx\sim c \int_0^T\I(x)dx<\infty.
 \]
 Concerning the second integral, Jensen inequality yields
 \begin{align*}
  \int_0^T\int_{x/2}^x\left|\frac{\NN(x)-\NN(y)}{x-y}\right|^2dy\,dx & = \int_0^T\int_{x/2}^x\left|\frac{1}{x-y}\int_y^x\I(s)ds\right|^2dy\,dx\\[.3cm]
                                                                     & \leq \int_0^T\int_{x/2}^x\frac{1}{x-y}\int_y^x\I^2(s)ds\,dy\,dx.
 \end{align*}
 Furthermore, since $\I$ is positive and convex by  Lemma \ref{lem:conv}, it is  $\I^2(s)\leq \I^2(x)+\I^2(y)$ for every $s\in[y,x]$, so that
 \[
  \int_0^T\int_{x/2}^x\frac{1}{x-y}\int_y^x\I^2(s)ds\,dy\,dx\leq \int_0^T\int_{x/2}^x(\I^2(y)+\I^2(x))dy\,dx.
 \]
 Now, noting that $\log^{-4}(1/y)\leq\log^{-4}(1/x)$ for all $y\in(x/2,x)$ and using again \eqref{eq:Iasympt},
 \[
  \int_0^T\int_{x/2}^x\I^2(y)dy\,dx \sim \int_0^T\int_{x/2}^x\frac{1}{y^2\log^4(\frac{1}{y})}dy\,dx \leq c\int_0^T\frac{1}{x\log^4(\frac{1}{x})}dx<\infty,
 \]
 whereas, on the other hand, 
 \[
  \int_0^T\int_{x/2}^x\I^2(x)dy\,dx\leq c \int_0^Tx\,\I^2(x)dx\sim c\int_0^T\frac{1}{x\log^4(\frac{1}{x})}dx<\infty.
 \]
 Thus
 \[
  \int_0^T\int_{x/2}^x\left|\frac{\NN(x)-\NN(y)}{x-y}\right|^2dy\,dx<\infty
 \]
 and, summing up, it is $[\NN]_{\dot{H}^{1/2}(0,T)}<\infty$, which concludes the proof.
\end{proof}

Therefore, we can claim the following result on the operator $I$ defined in \eqref{eq:op_I}.

\begin{teo}
\label{contr_lemma_2}
If $g\in H^{1/2}(0,T)\cap L^\infty(0,T)$, then
\begin{equation}
\label{eq:est_I_lim}
\|Ig\|_{H^{1/2}(0,T)} \leq c \cdot\max\{\|\NN\|_{H^{1/2}(0,T)},\NN(T)\} \left(\|g\|_{L^\infty(0,T)}+\|g\|_{H^{1/2}(0,T)}\right)
\end{equation}
(where $c>0$ is independent of $g$ and $T$). 
\end{teo}

\begin{proof}
 Since $\theta=1/2$, \eqref{eq:stimaf} does not hold. However, defining $g_e$ as in \eqref{eq:ext}, $\|g_e\|_{L^2(\R)}=\sqrt{2}\,\|g\|_{L^2(0,T)}$ and hence, arguing as in the proof of Theorem \ref{contr_lemma} (with $\nu=\mathcal{I}$), one can check that
 \begin{equation}
  \label{eq:est_I_ldue}
  \|Ig\|_{L^2(0,T)} \leq c \cdot \NN(T) \|g\|_{L^2(0,T)}.
 \end{equation}
 Then, it is left to estimate $[Ig]_{\dot{H}^{1/2}(0,T)}$. First, we note that for every $0<y<x<T$
 \[
  (Ig)(x)-(Ig)(y)=\int_y^x\I(s)g(x-s)ds-\int_0^y\I(s)(g(x-s)-g(y-s))ds.
 \]
 Hence,
 \begin{align}
  \label{eq:Naux_1}
  [Ig]_{\dot{H}^{1/2}(0,T)}^2\leq & \, 4\int_0^T\int_0^x\left|\frac{1}{x-y}\int_y^x\I(s)g(x-s)ds\right|^2dy\,dx+\nonumber\\[.3cm]
                                  & \, 4\int_0^T\int_0^x\left|\int_0^y\I(s)\frac{g(x-s)-g(y-s)}{t-s}ds\right|^2dy\,dx.
 \end{align}
 Now, one can easily see that, since $g\in L^\infty(0,T)$,
 \begin{align}
  \label{eq:Naux_2}
  4\int_0^T\int_0^x\left|\frac{1}{x-y}\int_y^x\I(s)g(x-s)ds\right|^2dy\,dx+ & \leq 4\|g\|_{L^\infty(0,T)}^2\int_0^T\int_0^x\left|\frac{\NN(x)-\NN(y)}{x-y}\right|^2dy\,dx\nonumber\\[.3cm]
                                                                            & = 2\|g\|_{L^\infty(0,T)}^2[\NN]_{\dot{H}^{1/2}(0,T)}^2\nonumber\\[.3cm]
                                                                            & \leq 2\|g\|_{L^\infty(0,T)}^2\|\NN\|_{\dot{H}^{1/2}(0,T)}^2
 \end{align}
 (where $\|\NN\|_{\dot{H}^{1/2}(0,T)}$ is finite by Lemma \ref{lem:N}). On the other hand, by Jensen inequality and monotonicity of $\NN$,
 \[
  \begin{array}{l}
   \displaystyle 4\int_0^T\int_0^x\left|\int_0^y\I(s)\frac{g(x-s)-g(y-s)}{t-s}ds\right|^2dy\,dx\leq\\[.7cm]
   \hspace{4cm}\displaystyle \leq4\NN(T)\int_0^T\int_0^x\int_0^y\I(s)\left|\frac{g(x-s)-g(y-s)}{x-y}\right|^2ds\,dy\,dx.
  \end{array}
 \]
 From the Fubini Theorem and a change of variables
 \[
  \begin{array}{l}
   \displaystyle \int_0^T\int_0^x\int_0^y\I(s)\left|\frac{g(x-s)-g(y-s)}{x-y}\right|^2ds\,dy\,dx=\\[.7cm]
   \hspace{4cm}\displaystyle =\int_0^T\I(s)\int_0^{T-s}\int_0^x\left|\frac{g(x-s)-g(y-s)}{x-y}\right|^2dy\,dx\,ds\\[.7cm]
   \hspace{4cm}\displaystyle \leq [g]_{\dot{H}^{1/2}(0,T)}^2\NN(T).
  \end{array}
 \]
 Consequently, 
 \begin{equation}
  \label{eq:Naux_3}
  4\int_0^T\int_0^x\left|\int_0^y\I(s)\frac{g(x-s)-g(y-s)}{t-s}ds\right|^2dy\,dx\leq 4\NN^2(T)\|g\|_{\dot{H}^{1/2}(0,T)}^2
 \end{equation}
 and plugging \eqref{eq:Naux_3} and \eqref{eq:Naux_2} into \eqref{eq:Naux_1},
 \[
  [Ig]_{\dot{H}^{1/2}(0,T)}^2\leq c\cdot\max\{\|\NN\|_{\dot{H}^{1/2}(0,T)}^2,\NN^2(T)\}\left(\|g\|_{L^\infty(0,T)}^2+\|g\|_{\dot{H}^{1/2}(0,T)}^2\right).
 \]
 Finally, combining with \eqref{eq:est_I_ldue}, \eqref{eq:est_I_lim} follows.
\end{proof}

\begin{rem}
 Note that, again, the contractive effect is preserved since both $\NN(T)$ and $\|\NN\|_{\dot{H}^{1/2}(0,T)}$ converges to $0$, as $T\to0$.
\end{rem}

\begin{rem}
 It is also worth stressing that Theorem \ref{contr_lemma_2} holds as well for any positive and integrable kernel $\nu$ whose integral function $N\in H^{1/2}(0,T)$ (such as, for instance, Abel kernels). However, since this is a very specific assumption, we preferred to present it in the relevant case of the Volterra kernel, where $\NN\in H^{1/2}(0,T)$ can be clearly shown, leaving to the reader further generalizations. 
\end{rem}

Finally, we show that a version of Theorem \ref{contr_lemma} holds also in $W^{1,1}(0,T)$. This result could seem disconnected from the framework of our paper, but nevertheless it further clarifies some specific features of $J_\nu$ and, then, we mention it for the sake of completeness.

\begin{teo}
 \label{contr_lemma_3}
 If $\nu>0$ and $\nu\in L^1(0,T)$, then
 \begin{equation}
  \label{Wunouno}
  \|J_\nu g\|_{W^{1,1}(0,T)}\leq N(T)\left(|g(0)|+\|g\|_{W^{1,1}(0,T)}\right),\quad\forall g\in W^{1,1}(0,T).
 \end{equation}
\end{teo}

\begin{proof}
 Recalling \eqref{J_nu} and arguing as in the proof of Theorem \ref{contr_lemma}, one finds that
 \[
  (J_\nu g)(x)=f(x)\quad\forall x\in[0,T],
 \]
 where
 \begin{equation}
  \label{eq-faux}
  f(x)=\int_\R\nu_T(x-s)g_T(s)\one_{[0,x]}(s)ds,
 \end{equation}
 with
 \[
  \nu_T(x):=\left\{
  \begin{array}{ll}
   \nu(x), & \text{if }x\in\;]0,T[\,,\\[.2cm]
   0,      & \text{if }x\in\R\backslash\;]0,T[\,,
  \end{array}
  \right.
  \quad\text{and}\quad
  g_T(x):=\left\{
  \begin{array}{ll}
   g(x), & \text{if }x\in\;]0,T[\,,\\[.2cm]
   0,    & \text{if }x\in\R\backslash\;]0,T[\,.
  \end{array}
  \right.
 \]
 Now, recalling that $\one_{[0,x]}(s)=H(s)-H(s-x)$ ($H$ denoting the Heaviside function) and that $\nu_T(s)H(-s)=0$, \eqref{eq-faux} reads
 \[
  f(x)=(\nu_T*g_T)(x).
 \]
 Then, by well known properties of the convolution product
 \[
  \|J_\nu g\|_{L^1(0,T)}\leq \|f\|_{L^1(\R)}\leq \|\nu_T\|_{L^1(\R)}\|g_T\|_{L^1(\R)}=N(T)\|g\|_{L^1(0,T)}.
 \]
 On the other hand, by \eqref{J_nu},
 \[
  (J_\nu g)(x)=\int_0^x\nu(s)g(x-s)ds
 \]
 and thus
 \[
  \frac{d}{dx}(J_\nu g)(x)=g(0)\nu(x)+\int_0^x\nu(x-s)\frac{d}{ds}g(s).
 \]
 Consequently, since $\frac{d}{ds}g\in L^1(0,T)$, arguing as before one finds \eqref{Wunouno}.
\end{proof}

\begin{rem}
 The proof of the previous theorem stresses a relevant difference between the cases $W^{1,1}$ and $H^1$, that arises in fact from the lack of further integrability (of ``power type'') of $\nu$. Indeed, if $\nu$ belongs only to $L^1(0,T)$ the additional assumption $g(0)=0$ clearly cannot be removed in $H^1$, whereas it is not necessary in $W^{1,1}$.
\end{rem}


\bigskip
\bigskip
\noindent

{\bf Acknowledgements.} R.C. and L.T. acknowledge the support of MIUR through the FIR grant 2013 ``Condensed Matter in Mathematical Physics (Cond-Math)'' (code RBFR13WAET). 


\def\cprime{$'$}

\end{document}